\documentclass[10pt,reqno,twoside]{amsart}
\usepackage{mathrsfs}
\renewcommand{\mathcal}{\mathscr}

\usepackage{amscd}
\usepackage{amsmath}
\usepackage{amsfonts}
\usepackage{amssymb}
\usepackage{color}

\newtheorem{theorem}{Theorem}
\newtheorem{lemma}[theorem]{Lemma}

\newtheorem{corollary}[theorem]{Corollary}

\newcommand{\R}{\mathbb R}
\newcommand{\N}{\mathbb N}

\newcommand{\CC}{\R^n\setminus }

\renewcommand{\geq}{\geqslant}
\renewcommand{\ge}{\geqslant}
\renewcommand{\leq}{\leqslant}
\renewcommand{\le}{\leqslant}

\begin{document}

\author{Bego\~na Barrios}
\address{
Departamento de Matem\'aticas\\
Universidad Aut\'onoma  de Madrid\\
Ciudad Universitaria de Cantoblanco\\
and
Instituto de Ciencias Matem\'{a}ticas, (ICMAT, CSIC-UAM-UC3M-UCM)\\
C/Nicol\'{a}s Cabrera 15, 28049-Madrid (Spain) }
\email{bego.barrios@uam.es}

\author{Alessio Figalli}
\address{
The University of Texas at Austin\\
Mathematics Dept. RLM 8.100\\
2515 Speedway Stop C1200\\
Austin, TX 78712-1202 (USA) } \email{figalli@math.utexas.edu}

\author{Enrico Valdinoci}
\address{
Dipartimento di Matematica\\
Universit\`a degli Studi di Milano\\
Via Cesare Saldini 50\\
20133 Milano (Italy) } \email{enrico@math.utexas.edu}

\title[Bootstrap regularity and nonlocal
minimal surfaces]{Bootstrap regularity\\
for integro-differential
operators\\
and its application\\
to nonlocal minimal surfaces}
\date{\today}

\begin{abstract}
We prove that $C^{1,\alpha}$ $s$-minimal surfaces are
of class $C^\infty$. For this, we develop a new bootstrap
regularity theory for solutions of integro-differential equations
of very general type, which we believe is of independent interest.
\end{abstract}

\maketitle

\tableofcontents

\section{Introduction}

Motivated by the structure of interphases arising in phase
transition models with long range interactions, in~\cite{CRS} the
authors introduced a nonlocal version of minimal surfaces. These
objects are obtained by minimizing a ``nonlocal perimeter'' inside
a fixed domain $\Omega$: fix $s \in (0,1)$, and given two sets
$A,B\subset \R^n$, let us define the interaction term
$$ L(A,B):=\int_A \int_B \frac{dx\,dy}{|x-y|^{n+s}}.$$
The nonlocal perimeter of $E$ inside $\Omega$ is defined as
\begin{multline*} {\rm Per}(E,\Omega,s):=L\big( E\cap\Omega,(\CC E)\cap\Omega\big)\\
+L\big( E\cap\Omega,(\CC E)\cap(\CC\Omega)\big)+L\big( (\CC
E)\cap\Omega,E\cap(\CC\Omega)\big).
\end{multline*}
Then nonlocal ($s$-)minimal
surfaces correspond to minimizers of the above functional with the
``boundary condition'' that $E\cap ({\CC}\Omega)$ is prescribed.

It is proved in~\cite{CRS} that ``flat $s$-minimal surface'' are
$C^{1,\alpha}$ for all $\alpha<s$, and in~\cite{CV1, ADM, CV2}
that, as $s\rightarrow 1^-$, the $s$-minimal surfaces approach the
classical ones, both in a geometric sense and in a
$\Gamma$-convergence framework, with uniform estimates
as~$s\rightarrow1^-$. In particular,
when $s$
is
sufficiently close to~$1$, they inherit some nice regularity
properties from the classical minimal surfaces  (see
also~\cite{Sou, SV1, SV2} for the relation between $s$-minimal
surfaces and the interfaces of some phase transition equations
driven by the fractional Laplacian).

On the other hand, all the previous literature only focused on
the~$C^{1,\alpha}$ regularity, and higher regularity was left as
an open problem. In this paper we address this issue, and we prove
that $C^{1,\alpha}$ $s$-minimal surfaces are indeed~$C^\infty$,
according to the following result\footnote{ Here and in the
sequel, we write~$x\in\R^n$ as~$x=(x',x_n)\in\R^{n-1}\times\R$.
Moreover, given~$r>0$ and~$p\in\R^n$, we define
$$K_r(p):= \{ x\in\R^n \,:\, |x'-p'|<r
{\mbox{ and }} |x_n-p_n|<r\}.$$ As usual, $B_r(p)$ denotes the
Euclidean ball of radius~$r$ centered at~$p$.
Given~$p'\in\R^{n-1}$, we set
$$B^{n-1}_r(p'):=\{ x'\in\R^{n-1} \,:\, |x'-p'|<r\}.$$
We also use the notation~$K_r:=K_r(0)$, $B_r:=B_r(0)$,
$B^{n-1}_r:= B^{n-1}_r(0)$.}:

\begin{theorem}\label{main} Let~$s\in(0,1)$, and~$\partial E$ be a~$s$-minimal surface
in~$K_R$ for some~$R>0$. Assume that
\begin{equation}\label{XC2}
\partial E\cap
K_R =\left\{ (x',x_n) \,:\, x' \in B_R^{n-1}{\mbox{ and }} x_n =u(x')\right\}
\end{equation}
for some
 $u:B^{n-1}_R\to \R$, with $u\in {C^{1,\alpha}} (B^{n-1}_R)$ for any $\alpha<s$ and~$u(0)=0$.
 Then
$$u\in C^{\infty}(B_{\rho}^{n-1})\quad\forall\,\rho\in (0,R).$$
\end{theorem}

The regularity result of Theorem~\ref{main} combined with
\cite[Theorem~6.1]{CRS} and \cite[Theorems 1, 3, 4, 5]{CV2},
implies also the following results (here and in the sequel, $\{
e_1, e_2,\dots, e_n\}$ denotes the standard Euclidean basis):

\begin{corollary}
Fix~$s_o\in(0,1)$. Let~$s\in(s_o,1)$ and~$\partial E$ be a~$s$-minimal
surface in~$B_R$ for some $R>0$. There exists~$\epsilon_\star>0$,
possibly depending on~$n$, $s_o$ and~$\alpha$, but independent
of~$s$ and $R$, such that if
$$ \partial E\cap B_R\subseteq \{ |x\cdot e_n|\le \epsilon_\star R\}
$$then~$\partial E\cap B_{R/2}$ is
a~$C^{\infty}$-graph in the~$e_n$-direction.\end{corollary}

\begin{corollary}
There exists~$\epsilon_o\in(0,1)$ such that
if~$s\in(1-\epsilon_o,1)$, then:
\begin{itemize}
\item If~$n\le 7$, any $s$-minimal surface is of class
$C^{\infty}$; \item If~$n=8$, any $s$-minimal surface is of class
$C^{\infty}$ except, at most, at countably many isolated points.
\end{itemize}
More generally, in any dimension $n$ there
exists~$\epsilon_n\in(0,1)$ such that if~$s\in(1-\epsilon_n,1)$
then any $s$-minimal surface is of class $C^{\infty}$ outside a
closed set~$\Sigma$ of Hausdorff dimension $n-8$.
\end{corollary}

Also, Theorem~\ref{main} here combined with Corollary~1
in~\cite{SVc} gives the following regularity result in the plane:

\begin{corollary}
Let~$n=2$. Then, for any~$s\in(0,1)$, any $s$-minimal surface is a smooth embedded curve of
class~$C^{\infty}$.
\end{corollary}

In order to prove Theorem~\ref{main} we establish in fact a very
general result about the regularity of integro-differential
equations, which we believe is of independent interest.

For this, we consider a kernel $K=K(x,w):\R^n
\times(\R^n\setminus\{0\})\rightarrow(0,+\infty)$ satisfying some
general structural assumptions. In the following, $\sigma \in
(1,2)$.

First of all, we suppose that $K$ is close to an autonomous kernel
of fractional Laplacian type, namely
\begin{equation}\label{ass silv}
\left\{
\begin{aligned}
&{\mbox{there exist~$a_0,r_0>0$ and $\eta \in (0,a_0/4)$
such that}}\\
&\left| \frac{|w|^{n+\sigma}
K(x,w)}{2-\sigma}-a_0\right|\le\eta
\qquad \forall \,x\in B_{1},\,w\in B_{r_0}\setminus\{0\}.
\end{aligned}
\right.
\end{equation}
Moreover, we assume that\footnote{Observe that we use~$|\cdot|$
both to denote the Euclidean norm of a vector and, for a
multi-index case $\alpha=(\alpha_1,\dots,\alpha_n)\in\N^n$, to
denote $|\alpha|:=\alpha_1+ \dots+\alpha_n$. However, the meaning
of $|\cdot|$  will always be clear from the context.}
\begin{equation}\label{sm 1}
\left\{
\begin{aligned}
&{\mbox{there exist $k \in \N\cup \{0\}$ and $C_k>0$ such that}}\\
&K\in C^{k+1}\big(B_{1}\times (\R^n\setminus\{0\})\big),\\
&\|\partial^\mu_x \partial^\theta_w K(\cdot,
w)\|_{L^\infty(B_{1})} \le \frac{C_k}{|w|^{n+\sigma+|\theta|}}\\
&\qquad
\qquad \forall\,\mu, \theta\in \N^n,\,|\mu|+|\theta|\le k+1,\,w\in \R^n\setminus\{0\}.\\
\end{aligned}
\right.
\end{equation}

Our main result is a ``Schauder regularity theory'' for
solutions\footnote{We adopt the notion of viscosity solution used
in~\cite{CScpam, CS2011, CS2012}.} of an integro-differential
equation. Here and in the sequel we use the notation
\begin{equation}\label{delta}
\delta u(x,w):= u(x+w)+u(x-w)-2u(x).\end{equation}

\begin{theorem}\label{boot}
Let~$\sigma\in (1,2)$, $k\in\N\cup\{0\}$, and $u\in
L^\infty(\R^n)$ be a viscosity solution of the equation
\begin{equation}
\label{eq:main}
\int_{\R^{n}}{K(x,w)\,\delta u(x,w)dw}=f(x,u(x))\qquad
\text{inside $B_1$,}
\end{equation}
with $f\in C^{k+1}(B_1\times\R)$. Assume that
$K:B_{1}\times (\R^n\setminus\{0\})\rightarrow (0,+\infty)$
satisfies assumptions \eqref{ass silv} and \eqref{sm 1} for the
same value of $k$.

Then, if $\eta$ in \eqref{ass silv} is sufficiently small (the
smallness being independent of $k$), we have $u\in
C^{k+\sigma+\alpha}(B_{1/2})$ for any $\alpha<1,$ and
\begin{equation}\label{3bis}
\| u\|_{C^{k+\sigma+\alpha}(B_{1/2})} \le C
\left(1+\|u\|_{L^\infty(\R^n)}+\|f\|_{L^{\infty}(B_{1}\times\R)}\right) ,
\end{equation}
where\footnote{As customary, when~$\sigma+\alpha\in(1,2)$ (resp.
$\sigma+\alpha>2$), by \eqref{3bis} we mean that~$u\in
C^{k+1,\sigma+\alpha-1}(B_{1/2})$ (resp. $u\in
C^{k+2,\sigma+\alpha-1}(B_{1/2})$). (To avoid any issue, we will
always implicitly assume that $\alpha$ is chosen different from
$2-\sigma$, so that $\sigma+\alpha\neq 2$.)} $C>0$ depends only on
$n$, $\sigma$, $k$, $C_k$, and~$\|f\|_{C^{k+1}(B_{1}\times\R)}$.
\end{theorem}
Let us notice that, since the right hand side in \eqref{eq:main} depends
on $u$, there is no uniqueness for such an equation.
In particular it is not enough for us to prove a-priori estimates for smooth
solutions and then argue by approximation, since we do not know if our solution
can be obtained as a limit of smooth solution.

We also note that, if in \eqref{sm 1} one replaces the $C^{k+1}$-regularity of $K$
with the $C^{k,\beta}$-assumption
\begin{equation}\label{new condition}
\|\partial^\mu_x
\partial^\theta_w K(\cdot, w)\|_{C^{0,\beta}(B_{1})} \le
\frac{C_k}{|w|^{n+\sigma+|\theta|}},
\end{equation}
for all $|\mu|+|\theta|\le k$, then we obtain the following:

\begin{theorem}\label{boot2}
Let~$\sigma\in (1,2)$, $k\in\N\cup\{0\}$, and $u\in
L^\infty(\R^n)$ be a viscosity solution of equation \eqref{eq:main}
with $f\in C^{k, \beta}(B_1\times\R)$. Assume that
$K:B_{1}\times (\R^n\setminus\{0\})\rightarrow (0,+\infty)$
satisfies assumptions \eqref{ass silv} and \eqref{new condition}
for the same value of $k$.

Then, if $\eta$ in \eqref{ass silv} is sufficiently small (the
smallness being independent of $k$), we have $u\in
C^{k+\sigma+\alpha}(B_{1/2})$ for any $\alpha<\beta,$ and
\begin{equation*}
\| u\|_{C^{k+\sigma+\alpha}(B_{1/2})} \le C
\left(1+\|u\|_{L^\infty(\R^n)}+\|f\|_{L^{\infty}(B_{1}\times\R)}\right) ,
\end{equation*}
where $C>0$ depends only on $n$, $\sigma$, $k$, $C_k$,
and~$\|f\|_{C^{k,\beta}(B_{1}\times\R)}$.
\end{theorem}

The proof of Theorem \ref{boot2} is essentially the same as the one of Theorem \ref{boot},
the only difference being that instead of differentiating
the equations (see for instance the argument in Section
\ref{section:uniforml}) one should use incremental quotients.
Although this does not introduce any major additional
difficulties, it makes the proofs longer and more tedious. Hence,
since the proof of Theorem \ref{boot} already contains all the
main ideas to prove also
Theorem \ref{boot2}, we will show the details of the proof only for Theorem \ref{boot}.

The paper is organized as follows: in the next section we prove
Theorem \ref{boot}, and then in Section \ref{section:main} we
write the fractional minimal surface equation in a suitable form
so that we can apply Theorems~\ref{boot} and~\ref{boot2}
to prove Theorem~\ref{main}.\\

\textit{Acknowledgements:} We wish to thank Guido De Philippis and Francesco Maggi for stimulating
our interest in this problem. We also thank Guido De Philippis for a
careful reading of a preliminary version of our manuscript,
and Nicola Fusco for kindly pointing out to us
a computational inaccuracy.
BB was partially supported by Spanish Grant MTM2010-18128. AF was partially supported by NSF Grant DMS-0969962.
EV was partially supported by ERC Grant 277749 and FIRB Grant A\&B.

\section{Proof of Theorem \ref{boot}}

The core in the proof of Theorem \ref{boot} is the step $k=0$,
which will be proved in several steps.

\subsection{Toolbox}   

We collect here some preliminary observations on scaled H\"older
norms, covering arguments, and differentiation of integrals that
will play an important role in the proof of Theorem~\ref{boot}.
This material is mainly technical, and the expert reader may go
directly to Section~\ref{TTT} at page~\pageref{TTT}.

\subsubsection{Scaled H\"older norms and coverings}

Given~$m\in\N$, $\alpha\in(0,1)$, $x\in\R^n$, and~$r>0$, we define
the $C^{m,\alpha}$-norm of a function~$u$ in~$B_r(x)$ as
$$ \|u\|_{C^{m,\alpha} (B_r(x))}:=
\sum_{|\gamma|\le m} \|D^\gamma u\|_{L^\infty (B_r(x))}
+\sum_{|\gamma|=m}\sup_{y\ne z\in B_r(x)}\frac{|D^\gamma u(y)-
D^\gamma u(z)|}{|y-z|^\alpha}.$$ For our purposes it is also
convenient to look at the following classical rescaled version of
the norm:
\begin{eqnarray*}
\|u\|^*_{C^{m,\alpha} (B_r(x))}&:=&\sum_{j=0}^m \sum_{|\gamma|=j}
r^{j} \|D^\gamma u\|_{L^\infty (B_r(x))}
\\&&+\sum_{|\gamma|=m} r^{m+\alpha}
\sup_{y\ne z\in B_r(x)}\frac{|D^\gamma u(y)- D^\gamma
u(z)|}{|y-z|^\alpha}.\end{eqnarray*} This scaled norm behaves
nicely under covering, as the next observation points out:

\begin{lemma}\label{co} Let~$m\in\N$, $\alpha\in(0,1)$, $\rho>0$, and~$x\in\R^n$. Fix $\lambda \in (0,1)$, and suppose that~$B_{\rho}(x)$ is covered
by finitely many balls~$\{B_{\lambda\rho/2}(x_k)\}_{k=1}^N$. Then,
there exists~$C_o>0$, depending only on $\lambda$ and~$m$, such that
$$ \|u\|^*_{C^{m,\alpha} (B_\rho (x))}\le C_o
\sum_{k=1}^N  \|u\|^*_{C^{m,\alpha} (B_{\lambda\rho}(x_k))}.$$
\end{lemma}

\begin{proof} We first observe that, if~$j\in \{0,\dots,m\}$ and~$|\gamma|=j$,
\begin{eqnarray*}
\rho^{j} \|D^\gamma u\|_{L^\infty (B_\rho(x))} &\le& \lambda^{-j}
(\lambda\rho)^{j} \max_{k=1,\dots,N}  \|D^\gamma u\|_{L^\infty
(B_{\lambda\rho} (x_k))}
\\ &\le& \lambda^{-m} \sum_{k=1}^N
(\lambda\rho)^{j} \|D^\gamma u\|_{L^\infty (B_{\lambda\rho} (x_k))}
\\ &\le& \lambda^{-m} \sum_{k=1}^N
\|u\|^*_{C^{m,\alpha} (B_{\lambda\rho} (x_k))}.
\end{eqnarray*}
Now, let~$|\gamma|=m$: we claim that
\begin{equation*}
\rho^{m+\alpha} \sup_{y\ne z\in B_{\rho}(x)}\frac{|D^\gamma u(y)-
D^\gamma u(z)|}{|y-z|^\alpha} \le2
\lambda^{-(m+\alpha)}\sum_{k=1}^N
\|u\|^*_{C^{m,\alpha} (B_{\lambda\rho}(x_k))}.
\end{equation*}
To check this, we take~$y, z\in B_{\rho}(x)$ with $y \neq z$ and
we distinguish two cases. If~$|y-z|< \lambda\rho/2$ we
choose~$k_o\in\{1,\dots,N\}$ such that~$y\in
B_{\lambda\rho/2}(x_{k_o})$. Then~$|z-x_{k_o}|\le
|z-y|+|y-x_{k_o}|<\lambda\rho$, which implies~$y,z\in
B_{\lambda\rho}(x_{k_o})$, therefore
\begin{eqnarray*}
\rho^{m+\alpha} \frac{|D^\gamma u(y)-D^\gamma u(z)|}{|y-z|^\alpha}
&\le& \rho^{m+\alpha} \sup_{\tilde y\ne \tilde z\in
B_{\lambda\rho}(x_{k_o})} \frac{|D^\gamma u(\tilde y)-D^\gamma
u(\tilde z)|}{
|\tilde y-\tilde z|^\alpha} \\
&\le&\lambda^{-(m+\alpha)}\| u\|^*_{ C^{m,\alpha} (B_{\lambda\rho}(x_{k_o}))}.
\end{eqnarray*}
Conversely, if~$|y-z|\geq\lambda\rho/2$, recalling that $\alpha \in (0,1)$
we have
\begin{eqnarray*}
\rho^{m+\alpha} \frac{|D^\gamma u(y)- D^\gamma
u(z)|}{|y-z|^\alpha} &\le& 2 \lambda^{-\alpha} \rho^{m}
{\|D^\gamma u\|_{L^\infty(B_\rho(x))}}\\
&\leq& 2 \lambda^{-\alpha} \rho^{m}\sum_{k=1}^{N}{\|D^{\gamma}u\|_{L^{\infty}(B_{\lambda\rho}(x_{k}))}}\\
&\leq&2
\lambda^{-(m+\alpha)}\sum_{k=1}^N\|u\|^{*}_{C^{m,\alpha}(B_{\lambda\rho}(x_{k}))}.
\end{eqnarray*}
This proves the claim and concludes the proof.
\end{proof}

Scaled norms behave also nicely in order to go from local to
global bounds, as the next result shows:

\begin{lemma}\label{Co2}
Let~$m\in\N$, $\alpha\in(0,1)$, and~$u\in C^{m,\alpha}(B_1)$.
Suppose that there exist $\mu \in (0,1/2)$ and $\nu\in (\mu,1]$ for which the following holds:
for any~$\epsilon>0$ there
exists~$\Lambda_\epsilon>0$ such that, for any~$x\in B_1$ and
any~$r\in (0,1-|x|]$, we have
\begin{equation}\label{X0}
\|u\|^*_{C^{m,\alpha}(B_{\mu r}(x))} \leq \Lambda_\epsilon +\epsilon
\|u\|^*_{C^{m,\alpha}(B_{\nu r}(x))}.
\end{equation}
Then there exist constants~$\epsilon_o$, $C>0$, depending only
on~$n$, $m$, $\mu$, $\nu$, and $\alpha$, such that
$$ \|u\|_{C^{m,\alpha}(B_{\mu})}\le  C\Lambda_{\epsilon_o}.$$
\end{lemma}

\begin{proof} First of all we observe that
\begin{equation*}
\|u\|^*_{C^{m,\alpha}(B_{\mu r}(x))}\le
\|u\|_{C^{m,\alpha}(B_{\mu r}(x))} \le \|u\|^*_{C^{m,\alpha}(B_1)}
\end{equation*}
because~$r\in(0,1)$, which implies that
$$
Q:=\sup_{{x\in B_1}\atop{r\in (0,1-|x|]}}
\|u\|^*_{C^{m,\alpha}(B_{\mu r}(x))}<+\infty.
$$
We now use a covering argument: pick $\lambda \in (0,1/2]$
to be chosen later, and fixed any~$x\in B_1$ and~$r\in
(0,1-|x|]$ we cover $B_{\mu r}(x)$ with finitely many balls
$\{B_{\lambda\mu r/2}(x_k)\}_{k=1}^N$, with $x_k \in B_{\mu r}(x)$, for
some~$N$ depending only on $\lambda$ and the dimension $n$. We now observe that, since $\mu<1/2$,
\begin{equation}
\label{eq:xk}
 |x_k|+ {r}/2\le |x_k-x|+|x|+{r}/2\le \mu{r}+
|x|+{r}/2< r+|x|\le 1.
\end{equation}
Hence, since $\lambda \leq 1/2$, we can use~\eqref{X0}
(with~$x=x_k$ and $r$ scaled to $\lambda r$) to obtain
\begin{eqnarray*}
\|u\|^*_{C^{m,\alpha}(B_{\lambda \mu r}(x_k))}
\leq \Lambda_\epsilon +\epsilon
\|u\|^*_{C^{m,\alpha}(B_{\lambda \nu r}(x_k))}.
\end{eqnarray*}
Then, using
Lemma~\ref{co} with~$\rho:=\mu r$ and $\lambda=\mu/(2\nu)$, 
and recalling \eqref{eq:xk} and the definition of $Q$, we get
\begin{eqnarray*}
\|u\|^*_{C^{m,\alpha}(B_{\mu r}(x))} &\leq & C_o
\sum_{k=1}^N \|u\|_{C^{m,\alpha}(B_{\lambda \mu r}(x_k))}^*\\
&\leq & C_o N \Lambda_\epsilon+C_o \epsilon \sum_{k=1}^N
\|u\|^*_{C^{m,\alpha}(B_{\lambda \nu r}(x_k))}\\
&=&C_o N \Lambda_\epsilon+C_o \epsilon \sum_{k=1}^N
\|u\|^*_{C^{m,\alpha}(B_{\mu r/2}(x_k))}\\
&\leq & C_oN\Lambda_\epsilon + \epsilon C_oN Q.
\end{eqnarray*}
Using the definition of $Q$ again, this implies
$$ Q \leq C_oN\Lambda_\epsilon + \epsilon C_oN Q, $$
so that, by choosing~$\epsilon_o:=1/(2C_oN)$,
$$ Q \leq 2C_o N \Lambda_{\epsilon_o}.$$
Thus we have proved that
$$
\|u\|_{C^{m,\alpha}(B_{\mu r}(x))}^* \le  2C_o N
\Lambda_{\epsilon_o} \qquad \forall\, x\in B_1, \,r\in (0,1-|x|],
$$
and the desired result follows setting~$x=0$ and~$r=1$.
\end{proof}

\subsubsection{Differentiating integral functions}

In the proof of Theorem~\ref{boot} we will need to differentiate,
under the integral sign, smooth functions that are either
supported near the origin or far from it. This purpose will be
accomplished in Lemmata~\ref{D} and~\ref{E}, after some technical
bounds that are needed to use the Dominated Convergence Theorem.

Recall the notation in~\eqref{delta}.

\begin{lemma}\label{D1}
Let~$r>r'>0$, $v\in C^3(B_r)$, $x\in B_{r'}$, $h\in\R$ with~$|h|<
(r-r')/2$. Then, for any~$w\in\R^n$ with~$|w|< (r-r')/2$, we have
$$ |\delta v(x+he_1,w)-\delta v(x,w)|\le |h|\, |w|^2
\|v\|_{C^3(B_r)} .$$
\end{lemma}

\begin{proof} Fixed~$x\in B_{r'}$ and~$|w|<
(r-r')/2$, for any~$h\in [(r'-r)/2,(r-r')/2]$ we
set~$g(h):=v(x+he_1+w)+v(x+he_1-w)-2v(x+he_1)$. Then
\begin{eqnarray*}&&  |g(h)-g(0)|\le |h| \sup_{|\xi|\le |h|} |g'(\xi)|
\\ &&\quad\le |h|\sup_{|\xi|\le |h|} \big|\partial_1 v(x+\xi e_1+w)
+\partial_1 v(x+\xi e_1-w)-2 \partial_1 v(x+\xi e_1)\big|.
\end{eqnarray*}
Noticing that~$|x+\xi e_1\pm w|\le r'+|h|+|w|<r$, a second order
Taylor expansion of~$\partial_1 v$ with respect to the
variable~$w$ gives
\begin{equation} \label{ed3}
 \big|\partial_1 v(x+\xi e_1+w)
+\partial_1 v(x+\xi e_1-w)-2 \partial_1 v(x+\xi e_1)\big| \le
|w|^2 \|\partial_1 v\|_{C^2(B_r)}.
\end{equation}
Therefore
\begin{eqnarray*}
|\delta v(x+he_1,w)-\delta v(x,w)|= |g(h)-g(0)| \le |h|\, |w|^2
\|v\|_{C^3(B_r)},
\end{eqnarray*}
as desired.
\end{proof}

\begin{lemma}\label{D1bis}
Let~$r>r'>0$, $v\in W^{1,\infty}(\R^n)$, $h\in\R$. Then, for
any~$w\in\R^n$,
$$ |\delta v(x+he_1,w)-\delta v(x,w)|\le 4|h| \|
\nabla v\|_{L^\infty(\R^n)}.$$
\end{lemma}

\begin{proof} It sufficed to proceed as in the proof of Lemma~{\ref{D1}}, but
replacing~\eqref{ed3} with the following estimate:
\begin{eqnarray*}
\big|\partial_1 v(x+\xi e_1+w) +\partial_1 v(x+\xi e_1-w)-2
\partial_1 v(x+\xi e_1)\big| \le 4 \|\partial_1 v\|_{L^\infty
(\R^n)} .\end{eqnarray*}
\end{proof}

\begin{lemma}\label{D}
Let~$\ell\in\N$, $r\in(0,2)$,~$K$ satisfy~\eqref{sm 1}, and~$U\in
C^{\ell+2}_0(B_r)$. Let~$\gamma=(\gamma_1,\dots,\gamma_n)\in\N^n$
with~$|\gamma|\le \ell \leq k+1$. Then
\begin{equation}\label{002}\begin{split}
&\partial^\gamma_x \int_{\R^n} K(x,w)\,\delta U(x,w)\,dw =
\int_{\R^n} \partial^\gamma_x\Big( K(x,w)\,\delta U(x,w)\Big)\,dw
\\ &\quad= \sum_{{{1\le i\le n}\atop{0\le \lambda_i\le
\gamma_i}}\atop{\lambda=(\lambda_1,\dots,\lambda_n)}} \left(
{\gamma_1}\atop{\lambda_1}\right)\dots \left(
{\gamma_n}\atop{\lambda_n}\right) \int_{\R^n} \partial^{\lambda}_x
K(x,w)\,\delta (\partial^{\gamma-\lambda}_x U)(x,w)\,dw
\end{split}\end{equation}
for any~$x\in B_r$.
\end{lemma}

\begin{proof} The latter equality follows from the standard product
derivation formula, so we focus on the proof of the first
identity. The proof is by induction over~$|\gamma|$.
If~$|\gamma|=0$ the result is trivially true, so we consider the
inductive step. We take~$x$ with~$r':=|x|<r$, we suppose
that~$|\gamma|\le \ell-1$ and, by inductive hypothesis, we know
that
$$g_\gamma(x):=
\partial^\gamma_x \int_{\R^n} K(x,w)\,\delta U(x,w)\,dw=
\int_{\R^n} \theta(x,w)\,dw$$ with
$$ \theta(x,w):=
\sum_{{{1\le i\le n}\atop{0\le \lambda_i\le
\gamma_i}}\atop{\lambda=(\lambda_1,\dots,\lambda_n)}} \left(
{\gamma_1}\atop{\lambda_1}\right)\dots \left(
{\gamma_n}\atop{\lambda_n}\right)
\partial^{\lambda}_x K(x,w)\,\delta
(\partial^{\gamma-\lambda}_x U)(x,w)\,dw.$$ By~\eqref{sm 1},
if~$0<|h|< (r-r')/2$ then
\begin{equation}\label{001}
|\partial^{\lambda}_x K(x+he_1,w)-
\partial^{\lambda}_x K(x,w)|\le {C_{|\lambda|+1}} |h|\,
|w|^{-n-\sigma}.
\end{equation}
Moreover, if~$|w|<(r-r')/2$, we can apply Lemma~\ref{D1}
with~$v:=\partial^{\gamma-\lambda}_x U$ and obtain
\begin{equation}\label{v0}
 |\delta (\partial^{\gamma-\lambda}_x U)(x+he_1,w)-
\delta(\partial^{\gamma-\lambda}_x U)(x,w)|\\
\le |h|\, |w|^2\|U\|_{C^{|\gamma-\lambda|+3}(B_r)}.
\end{equation}
On the other hand, by Lemma~\ref{D1bis} we obtain
$$ |\delta (\partial^{\gamma-\lambda}_x U)(x+he_1,w)-\delta
(\partial^{\gamma-\lambda}_x U)(x,w)|\le \,4\,|h|\, \|
\partial^{\gamma-\lambda}_x U\|_{C^1(\R^n)}.$$
All in all,
\begin{equation}\label{eq:1}\begin{split}
&|\delta (\partial^{\gamma-\lambda}_x U)(x+he_1,w)-\delta
(\partial^{\gamma-\lambda}_x U)(x,w)|
\\ &\qquad\le\,|h|\,
\|U\|_{C^{|\gamma-\lambda|+3}(\R^n)}\min\{4,|w|^2\}.
\end{split}\end{equation}
Analogously, a simple Taylor expansion provides also the bound
\begin{equation}\label{eq:2}
|\delta (\partial^{\gamma-\lambda}_x U)(x,w)|\le\,
\|U\|_{C^{|\gamma-\lambda|+2}(\R^n)}\min\{4,|w|^2\}.
\end{equation}
Hence, \eqref{sm 1}, \eqref{001}, \eqref{eq:1},  and \eqref{eq:2}
give
\begin{eqnarray*}
&& \big|
\partial^{\lambda}_x K(x+he_1,w)\,\delta
(\partial^{\gamma-\lambda}_x U)(x+he_1,w) -
\partial^{\lambda}_x K(x,w)\,\delta
(\partial^{\gamma-\lambda}_x U)(x,w)
\big| \\
&\le& \big|
\partial^{\lambda}_x K(x+he_1,w)\,\big[\delta
(\partial^{\gamma-\lambda}_x U)(x+he_1,w) -\delta
(\partial^{\gamma-\lambda}_x U)(x,w) \big]\big|
\\ &&+
\big|\big[
\partial^{\lambda}_x K(x+he_1,w)
-
\partial^{\lambda}_x K(x,w)
\big]\delta (\partial^{\gamma-\lambda}_x U)(x,w) \big|
\\ &\le& C_1 |h|\,\min\{|w|^{-n-\sigma},|w|^{2-n-\sigma} \},
\end{eqnarray*}
with~$C_1>0$ depending only on~$\ell$, $C_\ell$
and~$\|U\|_{C^{\ell+2}(\R^n)}$. As a consequence,
$$ |\theta(x+he_1,w)-\theta(x,w)|\le
C_2 |h|\,\min \{|w|^{-n-\sigma},|w|^{2-n-\sigma} \},$$ and, by the
Dominated Convergence Theorem, we get
\begin{eqnarray*}
\int_{\R^n}\partial_{x_1}\theta (x,w)\,dw &=& \lim_{h\rightarrow
0} \int_{\R^n}\frac{\theta (x+he_1,w)-\theta(x,w)}{h}\,dw
\\ &=& \lim_{h\rightarrow 0}\frac{g_\gamma(x+he_1)-g_\gamma (x)}{h}
\\ &=& \partial_{x_1}g_\gamma(x),
\end{eqnarray*}
which proves~\eqref{002} with~$\gamma$ replaced by~$\gamma+e_1$.
Analogously one could prove the same result with $\gamma$
replaced by~$\gamma+e_i$, concluding the inductive step.
\end{proof}

The differentiation under the integral sign in~\eqref{002} may
also be obtained under slightly different assumptions, as next
result points out:

\begin{lemma}\label{E}
Let~$\ell\in\N$, $R>r>0$. Let~$U\in C^{\ell+1}(\R^n)$ with~$U=0$
in~$B_{R}$. Let~$\gamma=(\gamma_1,\dots,\gamma_n)\in\N^n$
with~$|\gamma|\le \ell$. Then~\eqref{002} holds true for any~$x\in
B_r$.
\end{lemma}

\begin{proof} If~$x\in B_r$, $w\in B_{(R-r)/2}$
and~$|h|\le (R-r)/2$, we have that~$|x+w+h e_1|< R$ and so~$\delta
U(x+he_1,w)=0$. In particular
$$ \delta U(x+he_1,w)-\delta U(x,w)=0$$
for small~$h$ when~$w\in B_{(R-r)/2}$. This formula
replaces~\eqref{v0}, and the rest of the proof goes on as the one
of Lemma~\ref{D}.
\end{proof}

\subsubsection{Integral computations}

Here we collect some integral computations which will be used in
the proof of Theorem \ref{boot}.

\begin{lemma}
Let $v:\R^n \to \R$ be smooth and with all its derivatives
bounded. Let $x\in B_{1/4}$, and $\gamma$, $\lambda\in \N^n$,
with~$\gamma_i\ge\lambda_i$ for any~$i\in\{1,\dots,n\}$. Then
there exists a constant $C'>0$, depending only on $n$ and
$\sigma$, such that
\begin{equation}\label{A2 estimate}
\left|\int_{\R^n} \partial^{\lambda}_x K(x,w)\,\delta
(\partial^{\gamma-\lambda}_x v)(x,w)\,dw\right| \le C' \,
C_{|\gamma|} \,\| v\|_{C^{|\gamma-\lambda|+2} (\R^n)}
.\end{equation} Furthermore, if \begin{equation}\label{V001}
{\mbox{$v=0$ in $B_{1/2}$}}\end{equation} we have
\begin{equation}\label{A3 estimate}
\left|\int_{\R^n} \partial^{\lambda}_x K(x,w)\,\delta
(\partial^{\gamma-\lambda}_x v)(x,w)\,dw\right|\le C'\,
C_{|\gamma|}\, \| v\|_{L^\infty(\R^n)}.
\end{equation}
\end{lemma}

\begin{proof} By \eqref{sm 1} and \eqref{eq:2} (with $U=v$),
\begin{eqnarray*}
&& \int_{\R^n} \big|\partial^{\lambda}_x K(x,w)\big|\, \Big|
\,\delta (\partial^{\gamma-\lambda}_x v)(x,w)\Big|\,dw
\\ &&\le C_{|\lambda|} \left( \| v\|_{C^{|\gamma-\lambda|+2}(\R^n)}
\int_{B_2} |w|^{-n-\sigma+2}\,dw+ 4\|
v\|_{C^{|\gamma-\lambda|}(\R^n)} \int_{\R^n\setminus B_2}
|w|^{-n-\sigma}\,dw \right),
\end{eqnarray*}
which proves \eqref{A2 estimate}.

We now prove \eqref{A3 estimate}. For this we notice that, thanks
to~\eqref{V001}, $v(x+w)$ and $v(x-w)$ (and also their
derivatives) are equal to zero if $x$ and $w$ lie in $B_{1/4}$.
Hence, by an integration by parts we see that
\begin{eqnarray*}
&& \int_{\R^n} \partial^{\lambda}_x K(x,w)\,\delta
(\partial^{\gamma-\lambda}_x v)(x,w)\,dw
\\
&=& \int_{\R^n}\partial^{\lambda}_x
K(x,w)\,\partial^{\gamma-\lambda}_w \big[v(x+w)- v(x-w)\big]\,dw
\\ &=&
\int_{\R^n\setminus B_{1/4}} \partial^{\lambda}_x
K(x,w)\,\partial^{\gamma-\lambda}_w \big[v(x+w)- v(x-w)\big]\,dw
\\ &=& (-1)^{|\gamma-\lambda|}
\int_{\R^n\setminus B_{1/4}}
\partial^{\lambda}_x\partial^{\gamma-\lambda}_w K(x,w)\,
\big[v(x+w)- v(x-w)\big]\,dw.
\end{eqnarray*}
Consequently, by~\eqref{sm 1},
\begin{eqnarray*}
&& \left|\int_{\R^n} \partial^{\lambda}_x K(x,w)\,\delta
(\partial^{\gamma-\lambda}_x v)(x,w)\,dw\right|
\\ &&\le 2C_{|\gamma|}\, \| v\|_{L^\infty(\R^n)}
\int_{\R^n\setminus B_{1/4}} |w|^{-n-\sigma-|\gamma-\lambda|}
\,dw,
\end{eqnarray*}
proving \eqref{A3 estimate}.
\end{proof}

\subsection{Approximation by nicer kernels}
\label{TTT}

In what follows, it will be convenient
to approximate the solution $u$ of \eqref{eq:main} with smooth functions $u_\varepsilon$
obtained
by solving equations similar to \eqref{eq:main}, but with
kernels $K_\varepsilon$ which coincide with the fractional
Laplacian in a neighborhood of the origin. Indeed, this will allow
us to work with smooth functions, ensuring that in our
computations all integrals converge. We will then prove uniform
estimates on $u_\varepsilon$, which will give
the desired $C^{\sigma+\alpha}$-bound on $u$ by letting $\varepsilon \to 0$.

To simplify the notation, up to multiply both $K$ and $f$
by $1/a_0$, we assume without loss of generality that the constant $a_0$
in \eqref{ass silv} is equal to $1$.
\\

Let~$\eta\in C^{\infty}(\R^n)$ satisfy
$$
   \eta=\left\{\begin{array}{ll}
    1 &\quad\mbox{in }  B_{1/2}, \\
    0 &\quad\mbox{in } \R^n\setminus B_{3/4},
  \end{array}\right.
$$
and for any $\varepsilon,\delta>0$ set
$\eta_{\varepsilon}(w):=\eta\big(\frac{w}{\varepsilon}\big)$ for
any $\varepsilon>0$,
$\hat\eta_{\delta}(x):=\delta^{-n}\eta(x/\delta)$. Then we define
\begin{equation}\label{19bis}
K_{\varepsilon}(x,w):=\eta_{\varepsilon}(w)\frac{2-\sigma}{|w|^{n+\sigma}}+(1-\eta_{\varepsilon}(w))\hat
K_{\varepsilon}(x,w),
\end{equation}
where
\begin{equation}
\label{eq:hat K} \hat K_{\varepsilon}(x,w):=K(x,w)\ast \Big(
\hat\eta_{\varepsilon^2}(x)\hat\eta_{\varepsilon^2}(w)\Big),
\end{equation}
and
\begin{equation}\label{17b}
L_{\varepsilon}v(x):=\int_{\R^n}{K_{\varepsilon}(x,w)\,\delta
v(x,w) dw}.\end{equation} We also define
\begin{equation}
\label{eq:f eps}
f_{\varepsilon}(x):=f(x,u(x))\ast\hat\eta_{\varepsilon}(x).
\end{equation}
Note that we get a family $f_\varepsilon\in C^\infty(B_{1})$ such
that
$$ {\mbox{$\displaystyle\lim_{\varepsilon\to 0^+}{f_{\varepsilon}}=f$
uniformly in~$B_{3/4}$.}}$$ Finally, we define $u_{\varepsilon}\in
L^\infty(\R^{n})\cap C(\R^n)$ as the unique solution to the
following linear problem:
\begin{equation}\label{20bis}
\begin{array}{llll}\left\{\begin{matrix}
L_{\varepsilon}{u_{\varepsilon}}=
f_{\varepsilon}(x)&\quad\mbox{in }B_{3/4}\\
u_{\varepsilon}=u&\quad\mbox{in } \R^{n}\setminus B_{3/4}.
\end{matrix}\right.
\end{array}
\end{equation}
It is easy to check that the kernels $K_\varepsilon$ satisfy
\eqref{ass silv} and \eqref{sm 1} with constants independent of
$\varepsilon$ (recall that, to simplify the notation, we are assuming $a_0=1$).
Also, since $K$ satisfies assumption \eqref{sm 1} with
$k=0$ and the convolution parameter $\varepsilon^2$ in
\eqref{19bis} is much smaller than $\varepsilon$, the operators
$L_\varepsilon$ converge to the operator associated to $K$ in the
weak sense introduced in~\cite[Definition 22]{CS2011}. Indeed, let
$v$ a smooth function satisfying
\begin{equation}\label{v condition}
|v|\leq M \quad\mbox{in $\R^{n}$},\qquad |v(w)-v(x)-(w-x)\cdot\nabla v(x)|\leq M|x-w|^{2}\quad\forall\,w\in
B_{1}(x),
\end{equation}
for some $M>0$.
Then, from \eqref{sm 1} and \eqref{v condition}, it follows that
\begin{eqnarray}
&&\int_{\mathbb{R}^{n}}{\left|\eta_{\varepsilon}(w)\frac{2-\sigma}{|w|^{n+\sigma}}+(1-\eta_{\varepsilon}(\omega))\bigl(K(x,w)\ast\hat{\eta}_{\varepsilon^2}(x)\hat{\eta}_{\varepsilon^{2}}(w)\bigr)-K(x,w)\right||\delta v(x,w)|\,dw}\nonumber\\
&\le&\int_{\mathbb{R}^{n}}{\left(\eta_{\varepsilon}(w)\Big|\frac{2-\sigma}{|w|^{n+\sigma}}-K(x,w)\Big|+(1-\eta_{\varepsilon}(\omega))
\Big|K(x,w)\ast\hat{\eta}_{\varepsilon^2}(x)\hat{\eta}_{\varepsilon^{2}}(w)-K(x,w)\Big|\right)}\nonumber\\
&&\qquad\qquad \qquad\qquad\qquad \qquad\qquad\qquad \qquad \times|\delta v(x,w)|\,dw\nonumber\\
&\leq&\int_{B_{\varepsilon}}{C|w|^{2-n-\sigma}}+\int_{\R^n\setminus B_{\varepsilon}}{\bigl|K(x,w)\ast\hat{\eta}_{\varepsilon^2}(x)\hat{\eta}_{\varepsilon^{2}}(w)-K(x,w)\bigr|\,|\delta v(x,w)|\,dw}\nonumber\\
&\leq&
C\varepsilon^{2-\sigma}+I,
\label{ne}
\end{eqnarray}
with
$$I:=\int_{\R^n\setminus B_{\varepsilon}}{\bigl|K(x,w)\ast\hat{\eta}_{\varepsilon^2}(x)\hat{\eta}_{\varepsilon^{2}}(w)-K(x,w)\bigr|\,|\delta
v(x,w)|\,dw}.$$ By \eqref{sm 1}, \eqref{v condition}, and the fact
that $\sigma>1$, we have
\begin{eqnarray*}
I&=&\int_{\R^n\setminus B_{\varepsilon}}{\int_{B_{1}}{\int_{B_{1}}{\left|K(x-\varepsilon^2y,w-\varepsilon^2\tilde{w})\eta(y)\eta(\tilde{w})-K(x,w)\right|\,dy
\,d\tilde{w}\,|\delta v(x,w)|\,dw}}}\\
&\leq&\int_{\R^n\setminus B_{\varepsilon}}{\frac{C\varepsilon^2}{|w|^{n+1+\sigma}}\,|\delta v(x,w)|\,dw}\\
&\leq&C\int_{B_{1}\setminus B_{\varepsilon}}{\frac{\varepsilon^2}{|w|^{n-1+\sigma}}\,dw}+C\int_{\R^n\setminus B_{1}}{\frac{\varepsilon^2}{|w|^{n+1+\sigma}}\,dw}\\
&\leq &C(\varepsilon^{3-\sigma}+\varepsilon^{2}).
\end{eqnarray*}
Combining this estimate with \eqref{ne}, we get
$$\int_{\mathbb{R}^{n}}{\left|\eta_{\varepsilon}(w)\frac{2-\sigma}{|w|^{n+\sigma}}+(1-\eta_{\varepsilon}(\omega))(K(x,w)\ast\hat{\eta}_{\varepsilon^2}(x)\hat{\eta}_{\varepsilon^{2}}(w))-K(x,w)\right|\delta v(x,w)dw}\leq C\varepsilon^{2-\sigma},$$
where $C$ depends of $M$ and $\sigma$. Since $\sigma<2$ we conclude
that
$$\|L_{\varepsilon}-L\|\to 0 \quad\mbox{as $\varepsilon\to 0$.}$$
Thanks to this fact, we can repeat almost \textit{verbatim}\footnote{In order to repeat the
argument in the proof of \cite[Lemma 7]{CS2011} one needs to know
that the functions $u_\varepsilon$ are equicontinuous, which is a
consequence of \cite[{Lemmata 2 and 3}]{CS2011}. To be precise, to apply
\cite[Lemma 3]{CS2011} one would need the kernels to satisfy the
bounds $\frac{(2-\sigma)\lambda}{|w|^{n+\sigma}}\leq K_*(x,w)\leq
\frac{(2-\sigma)\Lambda}{|w|^{n+\sigma}}$ for all $w \neq 0$,
while in our case the kernel $K$ (and so also $K_\varepsilon$)
satisfies
\begin{equation}\label{Lo}
\frac{(2-\sigma)\lambda}{|w|^{n+\sigma}}\leq K(x,w)\leq
\frac{(2-\sigma)\Lambda}{|w|^{n+\sigma}}\qquad \forall\,|w|\leq
r_0
\end{equation}
with $\lambda:={a_0}-\eta$,
$\Lambda:=a_{0}+\eta$, and
$r_0>0$ (observe that, by our assumptions in \eqref{ass silv},
$\lambda \geq 3{a_0}/4$).

However this is not a big problem: if $v \in L^\infty(\R^n)$
satisfies
$$
\int_{\R^n} K_*(x,w)\,\delta v(x,w)\,dw = f(x) \qquad\text{in
$B_{3/4}$}
$$
for some kernel satisfying {\eqref{sm 1} and}
$\frac{(2-\sigma)\lambda}{|w|^{n+\sigma}}\leq K_*(x,w)\leq
\frac{(2-\sigma)\Lambda}{|w|^{n+\sigma}}$ only for $|w| \leq r_0$,
we define $K'(x,w):=\zeta(w)K_*(x,w) +
(2-\sigma)\frac{1-\zeta(w)}{|w|^{n+\sigma}}$, with $\zeta$ a
smooth cut-off function supported inside $B_{r_0}$, to get
$$
\int_{\R^n} K'(x,w)\,\delta v(x,w)\,dw = f(x) + \int_{\R^n}
[1-\zeta(w)]\left({-}K_*(x,w)+ \frac{2-\sigma}{|w|^{n+\sigma}}
\right)\,\delta v(x,w)\,dw.
$$
Since $1-\zeta(w)=0$ near the origin, {by assumption \eqref{sm 1}}, the second integral is
uniformly bounded as a function of $x$, so \cite[Lemma 3]{CS2011}
applied to $K'$ gives the desired equicontinuity.

Finally, the uniqueness for the boundary problem
$$
\left\{\begin{matrix}
\int_{\R^n} K(x,w) \,\delta v(x,w)\,dw=f(x,u(x))&\quad\mbox{in $B_{3/4}$,}\\
v=u\quad&\mbox{in $\R^{n}\setminus B_{3/4}$.}
\end{matrix}\right.
$$
follows by a standard comparison principle argument (see for
instance the argument used in the proof of \cite[Theorem
3.2]{BCF}). }
the proof of \cite[Lemma 7]{CS2011} to obtain the uniform convergence
\begin{equation}\label{Uu}
u_{\varepsilon}\to u \quad \text{on $\R^n$} \qquad\mbox{as $\varepsilon\to 0$.}
\end{equation}

\subsection{Smoothness of the approximate solutions}

We prove now that the functions $u_{\varepsilon}$ defined in the
previous section are of class $C^{\infty}$ inside a small ball
(whose size is uniform with respect to $\varepsilon$): namely,
there exists $r\in (0,1/4)$ such that, for any $m\in\N^{n}$
\begin{equation}\label{LE10} \| D^m u_{\varepsilon}
\|_{L^\infty(B_{r})}\le C \end{equation} for some positive
constant $C= C(m, \sigma, \varepsilon,\| u\|_{L^\infty(\R^n)},
\|f\|_{ L^\infty(B_{1}\times\R) })$.

For this, we observe that by~\eqref{19bis}
$$ \frac{2-\sigma}{|w|^{n+\sigma}}=K_\varepsilon(x,w)-
(1-\eta_{\varepsilon}(w))\hat K_{\varepsilon}(x,w)
+(1-\eta_\varepsilon(w))\,\frac{2-\sigma}{|w|^{n+\sigma}} ,$$ so for any  $x\in B_{1/4}$
\begin{eqnarray*}
\frac{2-\sigma}{2c_{n,\sigma}}(-\Delta)^{\sigma/2}u_{\varepsilon}(x)&=&
\int_{\R^n}{\frac{2-\sigma}{|w|^{n+\sigma}}\delta
u_{\varepsilon}(x,w)dw}\\
\qquad&=&f_{\varepsilon}(x)-\int_{\R^n}{(1-\eta_{\varepsilon}(w))\hat
K_{\varepsilon}(x,w)\, \delta
u_{\varepsilon}(x,w)dw}\\
\quad&&+\int_{\R^n}{(1-\eta_{\varepsilon}(w))\frac{2-\sigma}{|w|^{n+\sigma}}\,
\delta u_{\varepsilon}(x,w)dw}
\end{eqnarray*}
(here $c_{n,\sigma}$ is the positive constant that appears in the
definition of the fractional Laplacian, see e.g. \cite{PhS,
guida}). Then, for any $x\in B_{1/4}$ it follows that
\begin{eqnarray}
&& (-\Delta)^{\sigma/2}u_{\varepsilon}(x)
\nonumber\\
&=&d_{n,\sigma}\Big[f_{\varepsilon}(x)+\int_{\R^n}{(1-\eta_{\varepsilon}(w))\Big(\frac{2-\sigma}{|w|^{n+\sigma}}-\hat
K_{\varepsilon}(x,w)\Big)\delta
u_{\varepsilon}(x,w)dw}\Big]\nonumber\\
&=:&d_{n,\sigma}[f_{\varepsilon}(x)+h_\varepsilon(x)]\label{fund.eq}\\
&=:&d_{n,\sigma}g_\varepsilon(x).\nonumber
\end{eqnarray}
with $\displaystyle d_{n,\sigma}:=\frac{2c_{n,\sigma}}{2-\sigma}.$

Making some changes of variables we can rewrite  $h_\varepsilon$
as follows:
\begin{eqnarray}\label{BBB0}
\nonumber h_\varepsilon(x)&=&\int_{\R^n}
(1-\eta_{\varepsilon}(w-x))\Big(\frac{2-\sigma}{|w-x|^{n+\sigma}}-
\hat K_{\varepsilon}(x,w-x)\Big)u_{\varepsilon}(w) dw\\
&+&\int_{\R^n} (1-\eta_{\varepsilon}(x-w))
\Big(\frac{2-\sigma}{|w-x|^{n+\sigma}}-\hat
K_{\varepsilon}(x,x-w)\Big)u_{\varepsilon}(w) dw \nonumber
\\
&-&2 u_{\varepsilon}(x)\int_{\R^n}
(1-\eta_{\varepsilon}(w))\Big(\frac{2-\sigma}{|w|^{n+\sigma}}-\hat
K_{\varepsilon}(x,w)\Big) dw.
\end{eqnarray}
We now notice that ``the function $h_{{\varepsilon}}$ is locally
as smooth as $u_{\varepsilon}$'', is the sense that for any $m \in
\N$ and $U\subset B_{1/4}$ open we have
\begin{equation}\label{BL7}
\|h_{{\varepsilon}}\|_{C^{m}(U)} \leq
{C_{\varepsilon,\,m}}\left(1+\|u_{{\varepsilon}}\|_{C^{m}(U)}\right)
\end{equation}
for some constant $C_{\varepsilon,\,m}>0$. To see this observe that, in the first two
integrals, the variable $x$ appears only inside $\eta_\varepsilon$
and in the kernel $\hat K_{\varepsilon}$, and $\eta_{\varepsilon}$
is equal to $1$ near the origin. Hence the first two integrals are
smooth functions of $x$ (recall that $\hat K_{\varepsilon}$ is
smooth, see \eqref{eq:hat K}). The third term is clearly as
regular as $u_{\varepsilon}$ because the third integral is smooth
by the same reason as before. This proves \eqref{BL7}.

We are now going to prove that the functions $u_{\varepsilon}$
belong to $C^\infty(B_{1/5})$, with
\begin{equation}\label{LE1}
\|  u_{\varepsilon} \|_{C^m(B_{1/4-r_{m}})}\le
C(r_1,m,{\sigma},\varepsilon, \|
u_{\varepsilon}\|_{L^\infty(\R^n)},\|f\|_{ L^\infty(B_{1}\times\R)
})
\end{equation}
 for any $m\in\N$, where
$r_{m}:=1/20 - 25^{-m}$ (so that $1/4-r_m>1/5$ for any $m$).

To show this, we begin by observing that, since $\sigma\in(1,2)$,
by \eqref{fund.eq}, \eqref{BL7}, and \cite[Theorem 61]{CS2011}, we
have that $u_{\varepsilon}\in L^\infty(\R^{n})\cap
C^{1,\beta}(B_{1/4-r_1})$ for any $\beta<\sigma-1$ ($r_1=1/100$),
and
\begin{equation}\label{7B7A}
\|u_{\varepsilon}\|_{C^{1,\beta}(B_{1/4-r_1})}\leq
{C_{\varepsilon}}\Big(\|u\|_{L^{\infty}(\R^n)}+\|f\|_{L^{\infty}(B_{1}\times\R)}\Big).
\end{equation}
Now, to get a bound on higher derivatives, the idea would be to
differentiate~\eqref{fund.eq} and use again  \eqref{BL7} and
\cite[Theorem 61]{CS2011}. However we do not have $C^1$ bounds on
the function $u_{\varepsilon}$ outside
$B_{{1/4-r_1}}$, and therefore we can not apply
directly this strategy to obtain the $C^{2,\alpha}$ regularity of
the function $u_{\varepsilon}$.

To avoid this problem we follow the localization argument in
\cite[Theorem 13.1]{CScpam}: we take $\delta>0$ small (to be chosen)
and we consider a smooth cut-off function
$$
   \vartheta:=\left\{\begin{array}{ll}
    1 &\quad\mbox{in }  B_{1/4-(1+\delta)r_1}, \\
    0 &\quad\mbox{on } \R^n\setminus B_{1/4-r_1},
  \end{array}\right.
$$
and for fixed~$e\in S^{n-1}$ and $|h|<\delta r_1$ we define
\begin{equation}\label{Dv}
v(x):=\frac{u_{\varepsilon}(x+eh)-u_{\varepsilon}(x)}{|h|}.\end{equation}
{The function $v(x)$ is uniformly bounded in $B_{1/4-(1+\delta)r_1}$ because $u\in C^{1}(B_{1/4-r_1})$.}
We now
write~$v(x)=v_{1}(x)+v_{2}(x)$, being
$$v_{1}(x):=\frac{\vartheta u_{\varepsilon}(x+eh)-\vartheta u_{\varepsilon}(x)}{|h|}\quad\mbox{and}\quad v_{2}(x):=\frac{(1-\vartheta)u_{\varepsilon}(x+eh)-(1-\vartheta)u_{\varepsilon}(x)}{|h|}.$$
By \eqref{7B7A} it is clear that
$$
v_{1}\in L^{\infty}(\R^n)
$$
and that (recall that $|h|<\delta r_{1}$)
\begin{equation}\label{v1two}
v_1=v\qquad \mbox{inside $B_{1/4-(1+2\delta)r_1}$.}
\end{equation}
Moreover, for $x\in B_{1/4-(1+2\delta)r_1}$, using \eqref{fund.eq},
\eqref{eq:f eps}, and \eqref{BL7} we get
\begin{eqnarray}
\left|(-\Delta)^{\sigma/2}v_1(x)\right|&=&
\left|(-\Delta)^{\sigma/2}v(x)-(-\Delta)^{\sigma/2}v_{2}(x)\right|\nonumber\\
&=&\left|\frac{g_{{\varepsilon}}(x+eh)-g_{{\varepsilon}}(x)}{|h|}-(-\Delta)^{\sigma/2}v_{2}(x)\right|\nonumber\\
&\leq&
{C_{\varepsilon}}\Big(1+\|u_{\varepsilon}\|_{C^{1}(B_{1/4-r_1})}\Big)+\left|(-\Delta)^{\sigma/2}v_{2}(x)\right|.\label{M+}
\end{eqnarray}
Now, let us denote by
$K_o(y):=\frac{c_{n,\sigma}}{|y|^{n+\sigma}}$ the kernel of the
fractional Laplacian. Since for $x\in B_{1/4-(1+2\delta)r_1}$ and
$|y|<\delta r_{1}$ we have that $(1-\vartheta)u_{\varepsilon}(x\pm
y)=0$, it follows from a change of variable that
\begin{eqnarray*}
|(-\Delta)^{\sigma/2}v_{2}(x)|&\leq&\Big|\int_{\R^n}{(v_{2}(x+y)+v_{2}(x-y)-2v_{2}(x))K_o(y)dy}\Big|\\
&\le&\Big|\int_{\R^n}{\frac{(1-\vartheta)u_{\varepsilon}(x+y+eh)-(1-\vartheta)u_{\varepsilon}(x+y)}{|h|}K_o(y)dy}\Big|\\
&&+\Big|\int_{\R^n}{\frac{(1-\vartheta)u_{\varepsilon}(x-y+eh)-(1-\vartheta)u_{\varepsilon}(x-y)}{|h|}K_o(y)dy}\Big|\\
&\leq&\int_{\R^n}{(1-\vartheta)|u_{\varepsilon}|(x+y)\Big|\frac{K_o(y-eh)-K_o(y)}{|h|}\Big|dy}\\
&&+\int_{\R^n}{(1-\vartheta)|u_{\varepsilon}|(x-y)\Big|\frac{K_o(y-eh)-K_o(y)}{|h|}\Big|dy}\\
&\leq&\|u_{\varepsilon}\|_{L^{\infty}(\R^n)}\int_{\R^n\setminus
B_{{\delta r_1}} }{\frac{1}{|y|^{n+\sigma+1}}dy}\\
&\leq&C_\delta \|u_{\varepsilon}\|_{L^{\infty}(\R^n)}.
\end{eqnarray*}
Therefore, by \eqref{M+} we obtain
$$\big|(-\Delta)^{\sigma/2}v_{1}(x)\big|\leq
{C_{\varepsilon,\delta}}\Big(1+ \|u_{\varepsilon}\|_{C^{1}(B_{1/4-r_1})}
+\|u_{\varepsilon}\|_{L^{\infty}(\R^n)}\Big),\qquad x\in
{B_{1/4-(1+2\delta)r_1}},$$ and we can apply \cite[Theorem 61]{CS2011} to
get
that
$v_{1}\in C^{1,\beta}(B_{1/4-r_2})$ for any $\beta<\sigma-1$, with
$$\|v_{1}\|_{C^{1,\beta}(B_{1/4-r_{2}})}\leq
{C_{\varepsilon}}\Big(1+\|v_{1}\|_{L^{\infty}(\R^n)}+
\|u_{\varepsilon}\|_{C^{1}(B_{1/4}-r_1)}
+\|u_{\varepsilon}\|_{L^{\infty}(\R^n)}\Big),\,$$
provided $\delta>0$ was chosen sufficiently small
so that~{$r_2>(1+2\delta)r_1$}. By \eqref{Dv}, \eqref{v1two}, and
\eqref{7B7A}, this implies that $u_{\varepsilon}\in
C^{2,\beta}(B_{1/4-r_{2}})$, with
\begin{eqnarray*}
\|u_{\varepsilon}\|_{C^{2,\beta}(B_{1/4-r_{2}})}&\leq& {C_{\varepsilon}}\Big(1+
\|u_{\varepsilon}\|_{C^{1}(B_{1/4-r_1})}
+\|u_{\varepsilon}\|_{L^{\infty}(\R^n)}\Big)
\\ &\leq&
{C_\varepsilon}\Big(1 +\|u\|_{L^{\infty}(\R^n)}
+\|f\|_{L^{\infty}(B_{1}\times\R)}\Big).
\end{eqnarray*}
Iterating this argument we obtain~\eqref{LE1}, as desired.

\subsection{Uniform estimates and conclusion of the proof for $k=0$}
\label{section:uniforml} Knowing now that the functions
$u_\varepsilon$ defined by \eqref{20bis} are smooth inside $B_{1/5}$ (see \eqref{LE1}), our goal is
to obtain a-priori bounds independent of $\varepsilon$.

By \cite[Theorem 61]{CS2011} applied\footnote{As already observed in the footnote on
page~\pageref{Uu}, the fact that the kernel satisfies \eqref{Lo}
only for $w$ small is not a problem, and one can easily check that
\cite[Theorem 61]{CS2011} still holds in our setting. }
to $u$, we have that $u\in
C^{1,\beta}(B_{1-R_{1}})$ for any $\beta<\sigma-1$ and~$R_{1}>0$,
with
\begin{equation}\label{Be}
\|u\|_{C^{1,\beta}(B_{1-R_1})}\leq
C\left(\|u\|_{L^{\infty}(\R^n)}+\|f\|_{L^{\infty}(B_{1}\times\R)}\right).
\end{equation}
Then, for any
$\varepsilon$ sufficiently small,~$f_{\varepsilon}\in
C^1(B_{1/2})$ with
\begin{equation}\label{FE}\begin{split} & \|
f_{\varepsilon}\|_{C^1(B_{1/2})}\le C' \left(1+\| u\|_{C^{1}
(B_{1-R_1})}\right)\\
&\qquad\le C' C \left(1+\| u\|_{L^{\infty}(\R^n)}
+\|f\|_{L^{\infty}(B_{1}\times\R)}\right),
\end{split}\end{equation}
where~$C'>0$ depends on~$\|f\|_{C^1(B_{1}\times\R)}$ only.

Consider a cut-off function~$\tilde\eta$
which is $1$ inside $B_{1/7}$ and $0$ outside
$B_{1/6}$.

Then, recalling~\eqref{20bis}, we write the equation satisfied by
$u_{\varepsilon}$ as
\begin{equation*}
f_{\varepsilon}(x)= \int_{\R^{n}} K_{\varepsilon}(x,w)\,\delta
(\tilde\eta u_{\varepsilon})(x,w)dw +\int_{\R^{n}}
K_{\varepsilon}(x,w)\,\delta ((1-\tilde\eta) u_{\varepsilon})(x,w)dw,
\end{equation*}
and by differentiating it, say in direction~$e_1$, we obtain
(recall Lemmata~\ref{D} and~\ref{E})
\begin{eqnarray*}
\partial_{x_1} f_{\varepsilon}(x) &=&
\int_{\R^{n}} K_{\varepsilon}(x,w) \delta (\partial_{x_1}(\tilde\eta
u_{\varepsilon}))(x,w)dw
\\ &&\quad+\int_{\R^{n}} \partial_{x_1} \big[
K_{\varepsilon}(x,w)\delta ((1-\tilde\eta) u_{\varepsilon})(x,w)\big]
dw\\&&\quad +\int_{\R^{n}}
\partial_{x_1} K_{\varepsilon}(x,w)\delta (\tilde\eta u_{\varepsilon})(x,w)dw
\end{eqnarray*}for any~$x\in B_{1/5}$.
It is convenient to rewrite this equation as
$$ \int_{\R^{n}}  K_{\varepsilon}(x,w)
\delta (\partial_{x_1}(\tilde\eta
u_{\varepsilon}))(x,w)dw=A_1-A_2-A_3,$$ with
\begin{eqnarray*}
A_1 &:=& \partial_{x_1} f_{\varepsilon}(x),\\
A_2 &:=& \int_{\R^{n}} \partial_{x_1} K_{\varepsilon}(x,w)\delta (\tilde\eta u_{\varepsilon})(x,w)dw\\
A_3&:=& \int_{\R^{n}} \partial_{x_1} \big[
K_{\varepsilon}(x,w)\delta ((1-\tilde\eta) u_{\varepsilon})(x,w)\big]
dw.
\end{eqnarray*}
We claim that
\begin{equation}\label{A1233}
\|A_1-A_2-A_3\|_{L^\infty(B_{1/14})} \leq C
\left(1+\|u\|_{L^{\infty}(\R^n)}+\|u_{\varepsilon}\|_{C^{2}(B_{1/6})}\right)\end{equation}
with $C$ depending only on $\|f\|_{C^1(B_1\times\R)}$. 
To prove this, we first observe that by~\eqref{FE}
$$
\|A_1\|_{L^\infty(B_{1/14})} \leq C \left(1+
\|u\|_{L^{\infty}(\R^n)}\right).
$$
Also, since $| \partial_{x_1} \hat {K}_{\varepsilon}(x,w)| \leq C|w|^{-(n+\sigma)}$,\footnote{
This can be easily checked using the definition of $\hat {K}_{\varepsilon}$ and \eqref{sm 1}.
Indeed, because of the presence of the term $(1-\eta_{\varepsilon}(w))$
which vanishes for $|w|\leq \varepsilon/2$, one only needs to check that
$$
\int_{\R^n} |w-z|^{-n-\sigma} \hat{\eta}_{\varepsilon^2} (z) \,dz
\leq C|w|^{-n-\sigma} \qquad\text{for } |w|\geq \varepsilon/2,
$$
which is easy to prove (we leave the details to the reader).
}
by~\eqref{A2 estimate} (used with
$\gamma=\lambda:=(1,0,\dots,0)$ and $v:=\tilde\eta u_{\varepsilon}$) we get
$$
\|A_2\|_{L^\infty(B_{1/14})} \leq C\|\tilde\eta u_{\varepsilon}\|_{C^2(\R^n)}
\leq C
\|u_{\varepsilon}\|_{C^2(B_{1/6})},
$$
where we used that $\tilde\eta$ is supported in $B_{1/6}$.

Moreover, since~$(1-\tilde\eta)u_\varepsilon=0$ inside $B_{1/7}$, we can
use \eqref{A3 estimate} with $v:=(1-\tilde\eta)u_{\varepsilon}$ to
obtain
\begin{eqnarray*}
&& \left|\int_{\R^{n}} \partial_{x_1} K_{\varepsilon}(x,w)\,\delta
((1-\tilde\eta) u_{\varepsilon})(x,w)\,dw\right|
\\&&\qquad+
\left|\int_{\R^{n}} K_{\varepsilon}(x,w)\,\partial_{x_1}\delta
((1-\tilde\eta) u_{\varepsilon})(x,w)\, dw\right|
\\ &&\qquad\qquad\le
C\, C_{k}\, \| (1-\tilde\eta)u_{\varepsilon}\|_{L^\infty(\R^n)}
\end{eqnarray*}
for any $x\in B_{1/14}$,
which gives (note that, by an easy comparison principle,
$\|u_\varepsilon\|_{L^\infty(\R^n)} \leq
C(1+\|u\|_{L^\infty(\R^n)})$)
$$\|A_3\|_{L^\infty(B_{1/14})} \leq C(1+\|u\|_{L^\infty(\R^n)}).$$
The above estimates imply \eqref{A1233}.

Since $\partial_{x_1}(\tilde{\eta}u_{\varepsilon})$ is bounded
on the whole of~$\R^n$, by \eqref{A1233}
and~\cite[Theorem~61]{CS2011} we obtain that~$\partial_{x_1}(\tilde{\eta}
u_{\varepsilon})\in C^{1,\beta}(B_{1/14-R_2})$ for any
$R_{2}>0$, with
$$
\|\partial_{x_1}(\tilde{\eta}
u_{\varepsilon})\|_{C^{1,\beta}(B_{1/14-R_2})} \leq C
\left(1+\|u\|_{L^{\infty}(\R^n)}+\|u_{\varepsilon}\|_{C^{2}(B_{1/6})}\right),$$
which implies
\begin{equation}\label{29AA}
\|u_{\varepsilon}\|_{C^{2,\beta}(B_{1/15})} \leq C
\left(1+\|u\|_{L^{\infty}(\R^n)}+\|u_{\varepsilon}\|_{C^{2}(B_{1/6})}\right).
\end{equation}
To end the proof we need to reabsorb the $C^2$-norm on the right
hand side. To do this, we observe that by standard interpolation
inequalities (see for instance \cite[Lemma 6.35]{GT}), for any
$\delta\in(0,1)$ there exists $C_\delta>0$ such that
\begin{equation}\label{29BB}
\|u_{\varepsilon}\|_{C^2(B_{1/6})} \leq \delta
\|u_{\varepsilon}\|_{C^{2,\beta}(B_{1/5})}+C_\delta
\|u_{\varepsilon}\|_{L^\infty(\R^n)}.
\end{equation}
Hence, by~\eqref{29AA} and~\eqref{29BB} we obtain
\begin{equation}\label{29CC}
\|u_{\varepsilon}\|_{C^{2,\beta}(B_{1/15})} \leq C_\delta
(1+\|u\|_{L^\infty(\R^n)})
+C\delta\|u_{\varepsilon}\|_{C^{2,\beta}(B_{1/5})}.
\end{equation}
To conclude, one needs to apply the above estimates at every point
inside $B_{1/5}$ at every scale: for any $x \in B_{1/5}$, let $r>0$ be any
radius such that $B_r(x)\subset B_{1/5}$. Then we consider
\begin{equation}\label{sc}
v_{\varepsilon,r}^x(y):=u_{\varepsilon}(x+ry),
\end{equation}
and we observe that $v_{\varepsilon,r}^x$ solves an analogous
equation as the one solved by~$u_{\varepsilon}$ with the kernel
given by
$$K^{x}_{\varepsilon,r}(y,z):=r^{n+\sigma}K_{\varepsilon}(x+ry,rz)$$
and with right hand side
$$F_{\varepsilon,r}(y):=r^{\sigma}\int_{\R^n}{f(x+ry-\tilde{x},u(x+ry-\tilde{x}))\hat{\eta}_{\varepsilon}(\tilde{x})d\tilde{x}}.$$
We now observe that the kernels $K^{x}_{\varepsilon,r}$ satisfy assumptions \eqref{ass silv} and \eqref{sm 1} uniformly with
respect to $\varepsilon$, $r$, and $x$. Moreover, for  $|x|+r\leq
1/5$, and $\varepsilon$  small, we have
$$\|F_{\varepsilon,r}\|_{C^{1}(B_{1/2})}\leq r^{\sigma}C(1+\|u\|_{C^{1}(B_{3/4})}), $$
with $C>0$ depending on $\|f\|_{C^{1}(B_{1}\times\R)}$ only.
Hence, by \eqref{Be} this implies
$$\|F_{\varepsilon,r}\|_{C^{1}(B_{1/2})}\leq r^{\sigma}C\left(1+\|u\|_{L^{\infty}(\R^n)}+\|f\|_{L^{\infty}(B_{1}\times\R)}\right).$$
Thus, applying~\eqref{29CC} to $v_{\varepsilon,r}^x$ (by the
discussion we just made, the constants are all independent of
$\varepsilon$, $r$, and $x$) and scaling back, we get
\begin{eqnarray*}
\|u_{\varepsilon}\|^*_{C^{2,\beta}(B_{r/15}(x))} \leq
C_\delta
\left(1+\|u\|_{L^\infty(\R^n)}+\|f\|_{L^{\infty}(B_{1}\times\R)}\right)
+C\delta
\|u_{\varepsilon}\|^*_{C^{2,\beta}(B_{r/5}(x))}.
\end{eqnarray*}
Using now Lemma~\ref{Co2} inside $B_{1/5}$ with $\mu=1/15$, $\nu=1/5$, $m=2$, and $\Lambda_\delta=C_\delta
(1+\|u\|_{L^\infty(\R^n)}+\|f\|_{L^{\infty}(B_{1}\times\R)})$, we
conclude (observe that $1/15 \cdot 1/5=1/75$)
$$ \|u_{\varepsilon}\|_{C^{2,\beta}(B_{1/75})}\le  C\left(1+\|u\|_{L^\infty(\R^n)}+\|f\|_{L^{\infty}(B_{1}\times\R)}\right),$$
which implies
$$ \|u\|_{C^{2,\beta}(B_{1/75})}\le  C\left(1+\|u\|_{L^\infty(\R^n)}+\|f\|_{L^{\infty}(B_{1}\times\R)}\right)$$
by letting $\varepsilon \to 0$ (see \eqref{Uu}). Since
$\beta<\sigma-1$, this is equivalent to
$$ \|u\|_{C^{\sigma+\alpha}(B_{1/75})}\le  C
\left(1+\|u\|_{L^\infty(\R^n)}+\|f\|_{L^{\infty}(B_{1}\times\R)}\right),\,\quad\mbox{for any $\alpha<1$}.$$
A standard covering/rescaling argument completes the proof of
Theorem~\ref{boot} in the case $k=0$.

\subsection{The induction argument.}

We already proved Theorem~\ref{boot} in the case $k=0$.

We now show by induction that, for any $k \geq 1$,
\begin{equation}\label{induction}
\| u\|_{C^{k+\sigma+\alpha}(B_{1/2^{{3k+4}}})} \le C_k
\left(1+\|u\|_{L^\infty(\R^n)}+\|f\|_{L^{\infty}(B_{1}\times\R)}\right) ,
\end{equation}
for some constant~$C_k>0$: by a standard covering/rescaling
argument, this proves~\eqref{3bis} and so Theorem~\ref{boot}.
As we shall see, the argument is more or less identical to the case $k=0$.
To be fully rigorous, we should apply the regularization argument with the functions $u_\varepsilon$
as done in the previous step.
However, to simplify the notation and make the argument more transparent, we will skip the
regularization.

Define $g(x):=f(x,u(x))$, and consider a cut-off function~$\tilde\eta$
which is $1$ inside $B_{1/2^{3k+5}}$ and $0$ outside
$B_{1/2^{3k+4}}$.

By Lemmata~\ref{D} and~\ref{E} we differentiate the equation~$k+1$
times according to the following computation: first we observe
that, {since \eqref{induction} is true for $k-1$ and we can choose $\alpha \in (2-\sigma,1)$
so that $\sigma+\alpha>2$}, we deduce that $g\in C^{k+{1}}(B_{1/2^{3k+4}})$ with
\begin{equation}\label{gg}
 \|g\|_{C^{k+{1}}(B_{1/2^{3k+4}})}\leq C\left(1+ \| u\|_{C^{k+{1}}
(B_{1/2^{3k+4}})}\right){\leq C\left(\|u\|_{L^{\infty}(\R^n)}+\|f\|_{L^{\infty}(B_1\times\R^n)}\right)},
\end{equation}
with~$C>0$ depending on~$\|f\|_{C^{k+{1}}(B_{1}\times\R)}$ only. Now
we take~$\gamma\in\N^n$ with~$|\gamma|=k+{1}$ and we differentiate
the equation to obtain
\begin{eqnarray*}&& \partial^\gamma g(x)
\\ &=&
\sum_{{{1\le i\le n}\atop{0\le \lambda_i\le
\gamma_i}}\atop{\lambda=(\lambda_1,\dots,\lambda_n)}} \left(
{\gamma_1}\atop{\lambda_1}\right)\dots \left(
{\gamma_n}\atop{\lambda_n}\right) \int_{\R^n} \partial^{\lambda}_x
K(x,w)\,\delta (\partial^{\gamma-\lambda}_x (\tilde\eta u))(x,w)\,dw
\\ &+&
\sum_{{{1\le i\le n}\atop{0\le \lambda_i\le
\gamma_i}}\atop{\lambda=(\lambda_1,\dots,\lambda_n)}} \left(
{\gamma_1}\atop{\lambda_1}\right)\dots \left(
{\gamma_n}\atop{\lambda_n}\right) \int_{\R^n} \partial^{\lambda}_x
K(x,w)\,\delta (\partial^{\gamma-\lambda}_x (1-\tilde\eta)u)(x,w)\,dw
.\end{eqnarray*} Then, we isolate the term with~$\lambda=0$ in the
first sum:
$$
\int_{\R^n} K(x,w)\, \delta (\partial^\gamma_x(\tilde\eta u)) (x,w)\,dw
= A_1-A_2-A_3
$$
with
\begin{eqnarray*}
A_1 &:=& \partial^\gamma g(x),\\
A_2 &:=& \sum_{{{1\le i\le n}\atop{0\le \lambda_i\le
\gamma_i}}\atop{\lambda=(\lambda_1,\dots,\lambda_n) \ne0}} \left(
{\gamma_1}\atop{\lambda_1}\right)\dots \left(
{\gamma_n}\atop{\lambda_n}\right) \int_{\R^n} \partial^{\lambda}_x
K(x,w)\,\delta
(\partial^{\gamma-\lambda}_x (\tilde\eta u))(x,w)\,dw\\
A_3&:=& \sum_{{{1\le i\le n}\atop{0\le \lambda_i\le
\gamma_i}}\atop{\lambda=(\lambda_1,\dots,\lambda_n)}} \left(
{\gamma_1}\atop{\lambda_1}\right)\dots \left(
{\gamma_n}\atop{\lambda_n}\right) \int_{\R^n} \partial^{\lambda}_x
K(x,w)\,\delta (\partial^{\gamma-\lambda}_x (1-\tilde\eta)u)(x,w)\,dw
\end{eqnarray*}
We claim that
\begin{equation}\label{R}
\|A_1-A_2-A_3\|_{L^\infty(B_{1/2^{3k+6}})} \leq C
\left(1+\|u\|_{L^\infty(\R^n)}
+\|u\|_{C^{k+{2}}(B_{1/2^{3k+4}})}\right),
\end{equation}
Indeed, by the fact that~$|\gamma-\lambda|\le k$ we see that
\begin{equation}\label{gg2}
\begin{split}
\|A_2\|_{L^\infty(B_{1/2^{3k+6}})} 
& \leq C \, C_{k} \,\| \tilde\eta u\|_{C^{k+{2}} (\R^n)}\\
& \leq C \, C_{k} \,\| u\|_{C^{k+{2}} (B_{1/2^{3k+4}})}.
\end{split}\end{equation}
Furthermore, since $(1-\tilde\eta)u=0$ inside $B_{1/2^{3k+5}}$, we can
use \eqref{A3 estimate} with $v:=(1-\tilde\eta)u$ to obtain
$$ \|A_3\|_{L^\infty(B_{1/2^{3k+6}})}\le C \|u\|_{L^\infty(\R^n)}.$$
This last estimate, \eqref{gg}, and \eqref{gg2} allow us to
conclude the validity of~\eqref{R}.

Now, by~\cite[Theorem~61]{CS2011} applied to
$\partial^\gamma_x(\tilde\eta u)$ we get
$$
\|u\|_{C^{\sigma+k+\alpha}(B_{1/2^{3k+7}})} \leq C
\left(1+\|u\|_{C^{k+{2}}(B_{1/2^{3k+4}})}+\|u\|_{L^\infty(\R^n)}\right),$$
which is the analogous of \eqref{29AA} with $\sigma+\alpha=2+\beta$.
Hence, arguing as in the case $k=0$ (see the argument after \eqref{29AA}) we conclude that
$$ \|u\|_{C^{\sigma+k+\alpha}(B_{1/2^{{3(2k+1)+5}}})}\le
C\left(1+\|u\|_{L^\infty(\R^n)}+\|f\|_{L^{\infty}(B_{1}\times\R)}\right).$$
A covering argument shows the validity of \eqref{induction},
comcluding the proof of Theorem~\ref{boot}.

\section{Proof of Theorem~\ref{main}}
\label{section:main}

The idea of the proof is to write the fractional minimal surface
equation in a suitable form so that we can apply
Theorem~\ref{boot}.

\subsection{Writing the operator on the graph of~$u$}

The first step in our proof consists in writing the $s$-minimal
surface functional in terms of the function $u$ which (locally)
parameterizes the boundary of a set $E$. More precisely, we assume
that $u$ parameterizes $\partial E \cap K_R$ and that (without
loss of generality) $E \cap K_R$ is contained in the ipograph of
$u$. Moreover, since by assumption $u(0)=0$ and $u$ is of class
$C^{1,\alpha}$, up to rotating the system of coordinates (so that
$\nabla u(0)=0$) and reducing the size of $R$, we can also assume
that
\begin{equation}
\label{eq:small Lip}
\partial E \cap K_R\subset B_R^{n-1}\times [-R/8,R/8].
\end{equation}
Let $\varphi\in C^\infty(\R)$ be an even function satisfying
$$
   \varphi(t)=\left\{\begin{array}{ll}
    1 &\quad\mbox{if }  |t|\leq 1/4, \\
    0 &\quad\mbox{if } |t|\geq 1/2,
  \end{array}\right.
$$
and define the smooth cut-off functions
$$
\zeta_R(x'):=\varphi(|x'|/R)\qquad\eta_R(x):= \varphi(|x'|/R)
\varphi(|x_n|/R).
$$
Observe that
$$
\zeta_R= 1 \quad \text{in }B_{R/4}^{n-1},\qquad \zeta_R= 0\quad
\text{outside }B_{R/2}^{n-1},
$$
$$
\eta_R= 1\quad \text{in } K_{R/4},\qquad \eta_R= 0 \quad
\text{outside } K_{R/2}.
$$
We claim that, for any~$x\in\partial E \cap \left(B_{R/2}^{n-1}\times [-R/8, R/8]\right)$,
\begin{equation}\label{LHS-1}\begin{split}
& \int_{\R^{n}}\eta_R(y-x)\frac{\chi_E(y)-\chi_{\CC E}
(y)}{|x-y|^{n+s}}\,dy
\\ &\qquad=2\int_{\R^{n-1}}
F\left( \frac{u(x'-w')-u(x')}{|w'|}\right)
\frac{\zeta_R(w')}{|w'|^{n-1+s}}\,dw',\end{split}
\end{equation}
where
$$
F(t):=\int_0^t\frac{d\tau}{(1+\tau^2)^{(n+s)/2}}.
$$
Indeed, writing $y=x-w$ we have (observe that $\eta_R$ is even)
\begin{eqnarray}\label{JW}
\nonumber &&\int_{\R^n}\eta_R(y-x)\frac{\chi_E(y)-\chi_{\CC E}
(y)}{|x-y|^{n+s}}\,dy
\\ &=&\int_{\R^n}\eta_R(w)\frac{\chi_E(x-w)-\chi_{\CC E} (x-w)}{|w|^{n+s}}\,dw
\\ &=&\int_{\R^{n-1}}\zeta_R(w') \biggl[
\int_{-R/4}^{R/4} \frac{\chi_E (x-w)-\chi_{\CC E}(x-w)}{\Big(
1+(w_n/|w'|)^2\Big)^{(n+s)/2}}\,dw_n
\biggr]\frac{dw'}{|w'|^{n+s}}, \nonumber\end{eqnarray} where the
last equality follows from the fact that $\varphi(|w_n|/R)= 1$ for
$|w_n| \leq R/4$,
and that by \eqref{eq:small Lip} and by
symmetry the contributions of $\chi_E (x-w)$ and $\chi_{\CC
E}(x-w)$ outside $\{|w_n|\leq R/4\}$ cancel each other.

We now compute the inner integral: using the change
variable~$t:=w_n/|w'|$ we have
\begin{eqnarray*}
&&\int_{-R/4}^{R/4}\frac{\chi_{E} (x-w)}{\Big(
1+(w_n/|w'|)^2\Big)^{(n+s)/2}}\,dw_n
\\ &=&
\int_{u(x')-u(x'-w')}^{R/4}\frac{1}{\Big(
1+(w_n/|w'|)^2\Big)^{(n+s)/2}}\,dw_n
\\ &=& |w'| \int_{(u(x')-u(x'-w'))/|w'|}^{R/(4|w'|)}\frac{1}{\big(
1+t^2\big)^{(n+s)/2}}\,dt\\ &=& |w'| \biggl[ F\left(
\frac{R}{4|w|'}\right)-F\left(\frac{u(x')-u(x'-w')}{|w|'}\right)\biggr]
.\end{eqnarray*} In the same way,
\begin{eqnarray*}
&& \int_{-R/4}^{R/4}\frac{\chi_{\CC E} (x-w)}{\Big(
1+(w_n/|w'|)^2\Big)^{(n+s)/2}}\,dw_n \\ &&\qquad=|w'| \biggl[
F\left(\frac{u(x')-u(x'-w')}{|w|'}\right) -F\left(-
\frac{R}{4|w'|}\right)\biggr] .\end{eqnarray*} Therefore,
since~$F$ is odd, we immediately get that
\begin{eqnarray*}
\int_{-R/4}^{R/4}\frac{\chi_E (x-w)-\chi_{\CC E}(x-w)}{\Big(
1+(w_n/|w'|)^2\Big)^{(n+s)/2}}\,dw_n = 2|w'| F\left(
\frac{u(x'-w')-u(x')}{|w'|}\right),\end{eqnarray*} which together
with~\eqref{JW} proves~\eqref{LHS-1}.

Let us point out that to justify these computations in a pointwise
fashion one would need $u \in C^{1,1}(x)$ (in the sense of
\cite[Definition 3.1]{BCF2}). However, by using the viscosity
definition it is immediate to check that \eqref{LHS-1} holds in
the viscosity sense (since one only needs to verify it at points
where the graph of $u$ can be touched with paraboloids).

\subsection{The right hand side of the equation}

Let us define the function
\begin{equation}\label{smooth0}
\Psi_R(x):=
\int_{\R^{n}}\left[1-\eta_R(y-x)\right]\frac{\chi_E(y)-\chi_{\CC
E} (y)}{|x-y|^{n+s}}\,dy.
\end{equation}
Since $1-\eta_R(y-x)$ vanishes in a neighborhood of $\{x=y\}$, it
is immediate to check that the function $ \psi_R(z):=
\displaystyle\frac{1-\eta_R(z)}{|z|^{n+s}}$ is of class
$C^\infty$, with
$$
|\partial^\alpha \psi_R(z)| \leq
\frac{C_{|\alpha|}}{1+|z|^{n+s}}\qquad \forall\,\alpha\in \N^n.
$$
Hence, since $1/(1+|z|^{n+s}) \in L^1(\R^n)$ we deduce that
\begin{equation}\label{R7}
{\mbox{$\Psi_R \in C^\infty(\R^n)$, with all its derivatives
uniformly bounded.}}
\end{equation}

\subsection{An equation for~$u$ and conclusion}

By~\cite[Theorem 5.1]{CRS} we have that the equation
$$ \int_{\R^n} \frac{\chi_E(y)-\chi_{\CC E}(y)}{|x-y|^{n+s}}\,dy=0$$
holds in viscosity sense  for any~$x\in (\partial E)\cap K_R$.
Consequently, by~\eqref{LHS-1} and~\eqref{smooth0} we deduce that
$u$ is a viscosity solution of
$$
\int_{\R^{n-1}} F\left( \frac{u(x'-w')-u(x')}{|w'|}\right)
\frac{\zeta_R(w')}{|w'|^{n-1+s}}\,dw'=-\frac{\Psi_R(x',u(x'))}{2}
$$
inside $B_{R/2}^{n-1}$.
Since $F$ is odd, we can add the term $F\left(-\nabla u(x')\cdot \frac{w'}{|w'|}\right)$ inside the integral in the left hand side (since it integrates to zero), so the equation actually becomes 
\begin{multline}\label{eq1}
\int_{\R^{n-1}} \biggl[F\left( \frac{u(x'-w')-u(x')}{|w'|}\right)-F\left(-\nabla u(x')\cdot \frac{w'}{|w'|}\right)\biggr]
\frac{\zeta_R(w')}{|w'|^{n-1+s}}\,dw'\\
=-\frac{\Psi_R(x',u(x'))}{2}.
\end{multline}
We would like to apply the regularity
result from Theorem~\ref{boot2}, exploiting~\eqref{R7} to bound the
right hand side of~\eqref{eq1}. To this aim, using the Fundamental
Theorem of Calculus, we rewrite the left hand side in \eqref{eq1}
as
\begin{equation}
\label{eq:rewrite lhs} \int_{\mathbb{R}^{n-1}}{\bigl(u(x'-w')-u(x')+\nabla u(x')\cdot w'\bigr)
\frac{a(x',-w')\zeta_R(w')}{|w'|^{n+s}}}\,dw',
\end{equation}
where
$$
a(x',-w'):=\int_{0}^{1}\biggl(1+\left(t\frac{u(x'-w')-u(x')}{|w'|}-(1-t)\nabla u(x')\cdot \frac{w'}{|w'|}\right)^{2}\biggr)^{-({n+s})/{2}}\,dt.
$$
Now, we claim that
\begin{equation}\label{eq2}
\begin{split}
\int_{\mathbb{R}^{n-1}}\delta u(x',w')\, K_R(x',w')\,dw'
=-\Psi_R(x',u(x')) +A_R(x'),
\end{split}
\end{equation}
where
$$
K_R(x',w'):=
\frac{[a(x',w')+a(x',-w')]\zeta_R(w')}{2|w'|^{(n-1)+(1+s)}}
$$
and
$$
A_R(x'):=\int_{\mathbb{R}^{n-1}}[u(x'-w')-u(x')+\nabla u(x')\cdot
w'] \frac{[a(x',w')-a(x',-w')]\zeta_R(w')}{|w'|^{n+s}}\,dw'
$$
To prove~\eqref{eq2}
we introduce a short-hand notation:
we define
$$
u^{\pm}(x',w'):=u(x'\pm w')-u(x')\mp \nabla u(x')\cdot w',\qquad
a^\pm(x',w'):= a(x', \pm w')\frac{\zeta_R (w')}{|w'|^{n+s}},$$
while the integration over~${\mathbb{R}^{n-1}}$, possibly in the principal value sense, will be denoted by~$I[\cdot]$.
With this notation, and recalling~\eqref{eq:rewrite lhs}, it follows that
\eqref{eq1} can be written
\begin{equation}\label{eq:rewrite lhs.2}
-\frac{\Psi_R}2= I[u^- a^-].\end{equation}
By changing $w'$ with $-w'$ in the integral given by~$I$,
we see that
$$ I[u^+ a^+]=I[u^- a^-],$$
consequently~\eqref{eq:rewrite lhs.2} can be rewritten as
\begin{equation}
\label{eq:rewrite lhs2}
-\frac{\Psi_R}2= I[u^+ a^+].
\end{equation}
Notice also that
\begin{equation}\label{extra}
u^+ + u^- = \delta u, \qquad I[u^+(a^+-a^-)]=I[u^-(a^--a^+)].\end{equation}
{Hence,} adding~\eqref{eq:rewrite lhs.2} and~\eqref{eq:rewrite lhs2}, and {using \eqref{extra}},
we obtain
{
\begin{eqnarray*}
-\Psi_R &=& I[u^+ a^+]+I[u^- a^-]
\\ &=& \frac{1}{2}I[(u^+ + u^-) (a^++a^-)]+\frac{1}{2}I[(u^+ - u^-) (a^+-a^-)]
\\ &=& \frac{1}{2}I[\delta u\, (a^++a^-)]+\frac{1}{2}I[(u^+ - u^-) (a^+-a^-)]
\\ &=& \frac{1}{2}I[\delta u\, (a^++a^-)]-I[u^- (a^+-a^-)],
\end{eqnarray*}
}
which proves~\eqref{eq2}.

Now, to conclude the proof of Theorem~\ref{main}
it suffices to apply Theorem~\ref{boot2} iteratively: more precisely, let us start by assuming that
$u \in
C^{1,\beta}(B_{2r}^{n-1})$ for some $r\leq R/2$ and any $\beta<s$. Then, by the discussion above we get that $u$ solves
$$
\int_{\mathbb{R}^{n-1}}\delta u(x',w')\, K_{r}(x',w')\,dw'
=-\Psi_r(x',u(x')) +A_r(x)  \qquad \text{in }B_r^{n-1}.
$$
Moreover, one can easily check that the regularity of $u$ implies
that the assumptions of Theorem \ref{boot2} {with $k=0$} are satisfied with
$\sigma:=1+s$ and $a_0:= 1/(1-s)$. (Observe that \eqref{new
condition} holds since~$\|u\|_{C^{1,\beta}(B^{n-1}_{2r})}$.)
Furthermore, it is not difficult to check that,
for $|w'|\leq 1$,
$$
\left|[u(x'-w')-u(x')+\nabla u(x')\cdot
w']\,[a(x',w')-a(x',-w')]\right| \leq C|w'|^{2\beta+1},
$$
which implies that the integral defining $A_r$ is convergent by choosing
$\beta>s/2$.
Furthermore, a tedious computation (which we postpone to Subsection \ref{EEE} below)
shows that \begin{equation}\label{END} A_r \in
C^{2\beta-s}(B_{r}^{n-1}).\end{equation}

Hence, by Theorem \ref{boot2} with $k=0$ we deduce that $u \in
C^{1,2\beta}(B^{n-1}_{r/2})$.
But then this implies that $A_{r}\in
C^{4\beta-s}(B_{r/4}^{n-1})$ and so by Theorem \ref{boot2} again $u \in
C^{1,4\beta}(B^{n-1}_{r/8})$ for all $\beta <s$.
Iterating this argument infinitely many times\footnote{Note that,
once we know that $\|u\|_{C^{k,\beta}(B^{n-1}_{2r})}$ is bounded for some $k \geq 2$ and $\beta \in (0,1]$,
for any $|\gamma|\leq k-1$ we get
$$
\partial_x^\gamma A_r(x)=\int_{\mathbb{R}^{n-1}}
\partial_x^\gamma\bigl([u(x'-w')-u(x')+\nabla u(x')\cdot
w']\,[a(x',w')-a(x',-w')]\bigr)
\frac{\zeta_r(w')}{|w'|^{n+s}}\,dw',
$$
and exactly as in the case $k=0$ one shows that
$$
\left|\partial_x^\gamma\bigl([u(x'-w')-u(x')+\nabla u(x')\cdot
w']\,[a(x',w')-a(x',-w')]\bigr)\right| \leq C|w'|^{2\beta+1}\qquad \forall\,|w'|\leq 1,
$$
and that $A_r \in C^{k-1,2\beta-s}(B_{r}^{n-1})$.} 
we get that $u \in
C^{m}(B^{n-1}_{\lambda^mr})$ for some $\lambda>0$ small, for any $m\in\mathbb{N}$. Then, by a simple covering argument we obtain that $u \in
C^{m}(B^{n-1}_{\rho})$ for any $\rho<R$ and $m\in\mathbb{N}$, that is, $u$ is of class $C^{\infty}$ inside $B_{\rho}$ for any $\rho<R$. This completes the proof of
Theorem \ref{main}.

\subsection{H\"older regularity of $A_{R}$.}
\label{EEE}
We now prove \eqref{END},
i.e., if $u\in C^{1,\beta}(B^{n-1}_{2r})$ then $A_r\in C^{2\beta-s}(B^{n-1}_{r})$
($r\leq R/2$). For this we introduce the following notation:

$$U(x',w'):=u(x'-w')-u(x')+\nabla u(x')\cdot w'$$
and
$$p(\tau):=\frac{1}{ (1+\tau^2)^{\frac{n+s}{2}} }.$$
In this way we can write
\begin{equation}\label{57a}
a(x',\textcolor{red}{-}w')=\int_{0}^{1}{p\biggl(t\frac{u(x'-w')-u(x')}{|w'|}
-(1-t)\nabla u(x')\cdot \frac{w'}{|w'|}\biggr)\,dt}.
\end{equation}
Let us define
$$\mathcal{A}(x',w'):=a(x',w')-a(x',-w').$$
Then we have
$$A_r(x')=\int_{\R^{n-1}}{U(x',w')\frac{\mathcal{A}(x',w')}{|w'|^{n+s}}\zeta_r(w')\,dw'}.$$
To prove the desired H\"older condition of the function $A_r(x')$,  we first note that
$$U(x',w')=\int_0^1\bigl[\nabla u(x')-\nabla u(x'-tw')\bigr] dt\,\cdot w'.$$
Since $u\in C^{1,\beta}(B_{R}^{n-1})$ and $2r\leq R$, we get
\begin{equation}\label{fin1}
|U(x',w')-U(y',w')|\leq C\,
\min\{|x'-y'|^{\beta}|w'|,|w'|^{\beta+1}\},\quad\mbox{for $y'\in B_{r}^{n-1}$}
\end{equation}
and
\begin{equation}\label{fin2}
 |U(x',w')|\leq C|w'|^{\beta+1}.
\end{equation}
Therefore, from \eqref{fin1} and \eqref{fin2} it follows that,
for any $y'\in B_{r}^{n-1}$,
\begin{eqnarray}
 |A_{r}(x')-A_{r}(y')|&=&\bigg|\int_{\R^{n-1}}{\big(U(x',w')
\mathcal{A}(x',w')-U(y',w')\mathcal{A}(y',w')\big)\frac{\zeta_r(w')}{|w'|^{n+s}}\,dw'}\bigg|\nonumber\\
&\leq&C\int_{\R^{n-1}}{\min\{|x'-y'|^{\beta}|w'|,|w'|^{\beta+1}\}
\frac{|\mathcal{A}(x',w')|}{|w'|^{n+s}}\zeta_r(w')\,dw'}\nonumber\\
&+&C\int_{\R^{n-1}}{|w'|^{\beta+1}
\frac{|\mathcal{A}(x',w')-\mathcal{A}(y',w')|}{|w'|^{n+s}}\zeta_r(w')\,dw'}\nonumber\\
&=:&I_1(x',y')+I_2(x',y')\label{fin3}.
\end{eqnarray}
To estimate the last two integrals we define
$$\mathcal{A_*}(x',w'):=
a(x',w')-p\biggl(\nabla u(x')\cdot\frac{w'}{|w'|}\biggr).$$
With this notation
\begin{equation}\label{fin4}
 \mathcal{A}(x',w')=\mathcal{A_*}(x',w')-\mathcal{A_*}(x',-w').
\end{equation}
By \eqref{57a}
and \eqref{fin2}, since $|p'(t)| \leq C$ and $p$ is even, it follows that
\begin{eqnarray}\nonumber
&&|\mathcal{A_*}(x',-w')|\\
&\leq&\int_0^1\int_0^1 \biggl|\frac{d}{d\lambda} p\biggl(\lambda t\frac{u(x'-w')-u(x')}{|w'|}
-[\lambda (1-t)+(1-\lambda)]\nabla u(x')\cdot\frac{ w'}{|w'|}\biggr)\biggr|\,d\lambda\,dt\nonumber\\
&\leq &\int_0^1{t\frac{|U(x',w')|}{|w'|}
\biggl(\int_0^1{\bigg|p'\bigg(\lambda t\frac{U(x',w')}{|w'|}-\nabla u(x')\cdot\frac{w'}{|w'|}\bigg)\bigg|\,d\lambda}\biggr)\,dt}\nonumber\\
&\leq& C|w'|^{\beta}\label{fin5}
\end{eqnarray}
for all $|w'|\leq r$.

Estimating $\mathcal{A_*}(x',w')$ in the same way, by \eqref{fin4} and \eqref{fin5}, we get, for any $\beta>s/2$,
\begin{eqnarray}
 I_1(x',y')&\leq&C\int_{\R^{n-1}}{\min\{|x'-y'|^{\beta}|w'|,|w'|^{\beta+1}\}|w'|^{\beta-n-s}\zeta_r(w')\,dw'}\nonumber\\
&\le&C|x'-y'|^{\beta}\int_{|x'-y'|}^{r}{t^{\beta-s-1}dt}+\int_{0}^{|x'-y'|}{t^{2\beta-s-1}\,dt}\nonumber\\
&\le&C|x'-y'|^{2\beta-s}.\label{fin6}
\end{eqnarray}
On the other hand, to estimate $I_2$ we note that
\begin{eqnarray}
|\mathcal{A}(x',w')-\mathcal{A}(y',w')|&\leq&|\mathcal{A_*}(x',w')-\mathcal{A_*}(y',w')|\nonumber\\
&+&|\mathcal{A_*}(y',-w')-\mathcal{A_*}(x',-w')|.\label{fin7}
\end{eqnarray}
Hence, arguing as in \eqref{fin5} we have
\begin{eqnarray}
 &&|\mathcal{A_*}(x',-w')-\mathcal{A_*}(y',-w')|\nonumber\\
&\leq &
\int_0^1t\frac{|U(x',w')|}{|w'|}\int_0^1\bigg|p'\bigg(\lambda t\frac{U(x',w')}{|w'|}-\nabla u(x')\cdot\frac{w'}{|w'|}\bigg)\nonumber\\
&&\qquad\qquad\qquad\qquad\qquad\qquad -p'\bigg(\lambda t\frac{U(y',w')}{|w'|}
-\nabla u(y')\cdot\frac{w'}{|w'|}\bigg)\bigg|\,d\lambda\,dt\nonumber\\
&+&\int_0^1{t\frac{|U(x',w')-U(y',w')|}{|w'|}\int_0^1
{\bigg|p'\bigg(\lambda t\frac{U(y',w')}{|w'|}-\nabla u(y')\cdot\frac{w'}{|w'|}\bigg)\bigg|\,d\lambda}\,dt}\nonumber\\
&=:&I_{2,1}(x',y')+I_{2,2}(x',y').\label{fin8}
\end{eqnarray}
We bound each of these integrals separately. First, since $|p'(t)|\leq C$,
it follows immediately  from \eqref{fin1} that
\begin{eqnarray}
 I_{2,2}(x',y')\leq
C\min\{|x'-y'|^{\beta},|w'|^{\beta}\}.\label{fin9}
\end{eqnarray}
On the other hand, by \eqref{fin2}, \eqref{fin1}, and the fact that $u\in C^{1,\beta}(B_{R}^{n-1})$ and $p'$ is uniformly Lipschitz,
we get
\begin{eqnarray}
I_{2,1}(x',y')&\leq& C|w'|^{\beta}\bigg(\frac{|U(x',w')-U(y',w')|}{|w'|}+|\nabla u(x')-\nabla u(y')|\bigg)\nonumber\\
&\leq&C|w'|^{\beta}\Big(
\min\{|x'-y'|^{\beta},|w'|^{\beta}\}+|x'-y'|^{\beta}\Big)\nonumber\\
&\leq&C|w'|^{\beta}|x'-y'|^{\beta}.\label{fin10}
\end{eqnarray}
Then, assuming without loss of generality $r \leq 1$ (so that also $|x'-y'|\leq 1$),
 by \eqref{fin8}, \eqref{fin9}, and \eqref{fin10} it follows that
\begin{eqnarray}\nonumber
|\mathcal{A_*}(x',-w')-\mathcal{A_*}(y',-w')|&\leq&
C\bigg(\min\{|x'-y'|^{\beta},|w'|^{\beta}\}+|w'|^{\beta}|x'-y'|^{\beta}\bigg)
\\\label{fin11}
&\leq& C\min\{|x'-y'|^{\beta},|w'|^{\beta}\}.
\end{eqnarray}
As $|\mathcal{A_*}(y',w')-\mathcal{A_*}(x',w')|$ is bounded
in the same way, by \eqref{fin7}, we have
$$|\mathcal{A}(x',w')-\mathcal{A}(y',w')|\leq C\min\{|x'-y'|^{\beta},|w'|^{\beta}\}.$$
By arguing as in \eqref{fin6}, we get that, for any $s/2<\beta<s$,
\begin{eqnarray}
I_{2}(x',y')
&\leq &C\int_{\R^{n-1}}{|w'|^{\beta+1}\frac{\min\{|x'-y'|^{\beta},|w'|^{\beta}\}}{|w'|^{n+s}}\zeta_r(w')dw'}\nonumber\\
&\leq& C|x'-y'|^{2\beta-s}.\label{fin12}
\end{eqnarray}
Finally, by \eqref{fin3}, \eqref{fin6} and \eqref{fin12}, we conclude that
$$|A_{r}(x')-A_{r}(y')|\leq C|x'-y'|^{2\beta-s},\quad y'\in B_{r}^{n-1},$$
as desired.

\end{document}